\documentclass[12pt,a4paper]{article}

\usepackage{amssymb}
\usepackage{amsmath, amsthm}
\usepackage{arydshln}
\usepackage{setspace}
\usepackage{comment}
\usepackage[margin=1in]{geometry}

\usepackage{graphicx}        

\def\RR{{\bf R}}
\def\barRR{\overline{\bf R}}
\def\ZZ{{\bf Z}}

\newcommand{\opt}{\mathop{\rm opt} }
\newcommand{\dom}{\mathop{\rm dom} }
\newcommand{\Conv}{\mathop{\rm Conv} }

\newtheorem{Thm}{Theorem}[section]
\newtheorem{Prop}[Thm]{Proposition}
\newtheorem{Lem}[Thm]{Lemma}

\newtheorem{Rem}[Thm]{Remark}

\begin{document}

\title{Discrete Convex Functions on Graphs and \\Their Algorithmic Applications}
\author{Hiroshi HIRAI \\
	 Department of Mathematical Informatics, \\
	Graduate School of Information Science and Technology, \\
	The University of Tokyo, Tokyo, 113-8656, Japan, \\
	\texttt{hirai@mist.i.u-tokyo.ac.jp}}
\maketitle

\begin{abstract}
	The present article is an exposition of 
	a theory of discrete convex functions on certain graph structures, 
	developed by the author in recent years. 
	This theory is a spin-off of discrete convex analysis by Murota, 
	and is motivated by combinatorial dualities in multiflow problems 
	and the complexity classification of facility location problems on graphs.
	We outline the theory and algorithmic applications 
	in  combinatorial optimization problems.  
\end{abstract}

\section{Introduction}
\label{sec:intro}
The present article is an exposition of a theory of discrete convex functions on certain graph structures, developed by the author in recent years. This theory is viewed as a spin-off of {\em Discrete Convex Analysis (DCA)}, which is a theory of convex functions on the integer lattice 
and has been developed by Murota and his collaborators 
in the last 20 years; see \cite{Murota98,MurotaBook,MurotaDevelop} and \cite[Chapter VII]{FujiBook}.
Whereas the main targets of DCA are matroid-related optimization (or submodular optimization),
our new theory is motivated by combinatorial dualities arising from multiflow problems~\cite{HH11folder} and the complexity classification of certain facility location problems on graphs~\cite{HH150ext}.

The heart of our theory is analogues of  
{\em L$^\natural$-convex functions}~\cite{FT90,FM00} 
for certain graph structures, 
where L$^\natural$-convex functions are 
one of the fundamental classes of discrete convex functions on $\ZZ^n$, 
and play primary roles in DCA.
These analogues are inspired by
the following intriguing properties of L$^\natural$-convex functions: 

\begin{itemize}
\item An L$^\natural$-convex function is (equivalently) 
defined as a function $g$ on $\ZZ^n$
satisfying a discrete version of the convexity inequality, 
called the {\em discrete midpoint convexity}:
\begin{equation}\label{eqn:dmc}
g(x) + g(y) \geq g (\lfloor (x+y)/2 \rfloor ) + 
g( \lceil (x+y)/2 \rceil) \quad (x,y \in \ZZ^n),
\end{equation}
where $\lfloor \cdot \rfloor$ and $\lceil \cdot \rceil$ are operators rounding down and up, respectively, the fractional part of each component.

\item L$^\natural$-convex function $g$ is {\em locally submodular} in the following sense: For each $x \in \ZZ^n$, 
the function on $\{0,1\}^n$ defined by $u \mapsto g(x+u)$
is submodular.
\item
Analogous to ordinary convex functions, 
L$^\natural$-convex function $g$ enjoys an optimality criterion of a local-to-global type:
If $x$ is not a minimizer of $g$, then there exists $y \in (x + \{0,1\}^n) \cup (x - \{0,1\}^n)$ 
with $g(y) < g(x)$.
\item
This leads us to a conceptually-simple minimization algorithm, 
called the {\em steepest descent algorithm (SDA)}:
For $x \in \ZZ^n$, find a (local) minimizer $y$ of $g$ over $(x + \{0,1\}^n) \cup (x - \{0,1\}^n)$
via submodular function minimization (SFM). 
If $g(y) = g(x)$, then $x$ is a (global) minimizer of $g$.
If $g(y) < g(x)$, then let $x := y$, and repeat.
\item
The number of iterations of SDA is sharply bounded by a certain $l_{\infty}$-distance between the initial point and minimizers~\cite{MurotaShioura14}.
\item L$^\natural$-convex function $g$ is extended to 
a convex function $\overline{g}$ on $\RR^n$ 
via the {\em Lov\'asz extension}, and this convexity property 
characterizes the  L$^\natural$-convexity.
\end{itemize}
Details are given in~\cite{MurotaBook}, and
a recent survey~\cite{Shioura17L-convex_survey} is also a good source of  L$^\natural$-convex functions.

We consider certain classes of graphs $\Gamma$ that canonically 
define functions (on the vertex set of $\Gamma$) having analogous properties, 
which we call {\em L-convex functions on $\Gamma$} (with $^\natural$ omitted).
The aim of this paper is to explain these L-convex functions and
their roles in combinatorial optimization problems, 
and to demonstrate the SDA-based algorithm design.

Our theory is parallel with recent developments 
in generalized submodularity and {\em valued constraint satisfaction problem (VCSP)}~\cite{KTZ13,TZ16ACM,ZivnyBook}.
Indeed, localizations of these L-convex functions give rise to 
a rich subclass of generalized submodular functions that
include $k$-submodular functions~\cite{HK12} and submodular functions on diamonds~\cite{FKMTT14,Kuivinen11}, 
and are polynomially minimizable in the VCSP setting.

The starting point is the observation that
if $\ZZ^n$ is viewed as a grid graph (with order information),
some of the above properties of L$^\natural$-convex functions 
are still well-formulated. Indeed, 
extending the discrete midpoint convexity to trees, 
Kolmogorov~\cite{Kolmogorov11} 
introduced a class of discrete convex functions, 
called {\em tree-submodular functions}, on the product of rooted trees, 
and showed several analogous results. 
In Section~\ref{sec:gridstructures}, 
we discuss L-convex functions on such grid-like structures.
We start by considering a variation of tree-submodular functions, 
where the underlying graph is the product of zigzagly-oriented trees.
Then we explain that a theory of L-convex functions is naturally developed on
a structure, known as {\em Euclidean building}~\cite{BuildingBook}, 
which is a kind of an amalgamation of $\ZZ^n$ 
and is a far-reaching generalization of a tree.
Applications of these L-convex functions are given in Section~\ref{sec:applications}.
We outline SDA-based efficient combinatorial algorithms 
for two important multiflow problems.
In Section~\ref{sec:oriented_modular}, 
we explain L-convex functions on a more general class of graphs, 
called {\em oriented modular graphs}.
This graph class emerged from the complexity classification 
of the {\em minimum 0-extension problem}~\cite{Kar98a}.
We outline that the theory of L-convex functions leads to
a solution of this classification problem.
This was the original motivation of our theory.

The contents of this paper are based 
on the results in papers~\cite{HH14extendable,HH15node_multi,HH150ext,HH16L-convex}, 
in which further details and omitted proofs are found.

Let $\RR$, $\RR_+$, $\ZZ$, and $\ZZ_+$ denote 
the sets of reals, nonnegative reals, integers, and nonnegative integers, respectively.
In this paper, $\ZZ$ is often regarded as an infinite path 
obtained by adding 
an edge to each consecutive pair of integers.
Let $\overline{\RR} := \RR \cup \{\infty \}$, 
where $\infty$ is the infinity element treated as $a < \infty$,
$a + \infty = \infty$ for $a \in \RR$, and $\infty + \infty = \infty$.
For a function 
$g: X \to \overline{\RR}$ on a set $X$, 
let $\dom g := \{ x \in X \mid g(x) < \infty \}$.

\section{L-convex function on grid-like structure}
\label{sec:gridstructures}

In this section, we discuss L-convex functions on grid-like structures.
In the first two subsections (Sections~\ref{subsec:tree-grids} and \ref{subsec:skew}), 
we consider specific underlying graphs (tree-product and twisted tree-product) that 
admit analogues of discrete midpoint operators 
and the corresponding L-convex functions.
In both cases, the local submodularity 
and the local-to-global optimality criterion 
are formulated in a straightforward way (Lemmas~\ref{lem:L-optimality1}, \ref{lem:local_1}, \ref{lem:L-optimality2}, and \ref{lem:local_2}).
In Section~\ref{subsec:SDA}, 
we introduce the steepest descent algorithm (SDA) in a generic way, 
and present the iteration bound (Theorem~\ref{thm:l_inf_bound}).
In Section~\ref{subsec:building}, 
we explain that the L-convexity is naturally generalized
to that in Euclidean buildings of type C, and that 
the Lov\'asz extension theorem
can be generalized via 
the geometric realization of Euclidean buildings and CAT(0)-metrics (Theorem~\ref{thm:convex}).

We use a standard terminology on posets (partially ordered sets) and lattices; 
see e.g., \cite{Gratzer}.
Let $P = (P, \preceq)$ be a poset. 
For an element $p \in P$, 
the {\em principal ideal} $I_p$ is the set of elements 
$q \in P$ with $q \preceq p$, and the 
{\em principal filter} 
$F_p$ is the set of elements $q  \in P$ with $p \preceq q$.
For $p,q \in P$ with $p \preceq q$, 
let $\max \{p,q\} := q$ and $\min \{p,q\} := p$.
A {\em chain} is a subset $X \subseteq P$ 
such that for every $p,q \in X$ it holds that $p \preceq q$ or $q \preceq p$.
For $p \preceq q$, the {\em interval} $[p,q]$ of $p,q$ 
is defined as $[p,q]:= \{ u \in P \mid p \preceq u \preceq q \}$.
For two posets $P$, $P'$, 
the direct product $P \times P'$ becomes  
a poset by the direct product order $\preceq$ defined by $(p,p') \preceq (q,q')$ 
if and only if $p \preceq q$ in $P$ and $p' \preceq q'$ in $P'$.

For an undirected graph $G$, 
an edge joining vertices $x$ and $y$ is denoted by $xy$. 
When $G$ plays a role of
the domain of discrete convex functions, 
the vertex set $V(G)$ of $G$ is also denoted by $G$.
Let $d = d_G$ denote the shortest path metric on $G$.
We often endow $G$ with an edge-orientation, 
where $x \to y$ means that edge $xy$ is oriented from $x$ to $y$.
For two graphs $G$ and $H$, 
let $G \times H$ denote the Cartesian product of $G$ and $H$, 
i.e., the vertices are all pairs of vertices of $G$ and $H$ and  
two vertices $(x,y)$ and $(x',y')$ are adjacent if 
and only if $d_G(x,x') + d_H(y,y') = 1$.

\subsection{L-convex function on tree-grid}
\label{subsec:tree-grids}

Let $G$ be a tree. 
Let $B$ and $W$ denote the color classes of $G$ viewed as 
a bipartite graph.
Endow $G$ with a zigzag orientation so that $u \to v$ if and only 
if $u \in W$ and $v \in B$.
This orientation is acyclic.
The induced partial order on $G$ is denoted by $\preceq$, 
where $v \leftarrow u$ is interpreted as $v \preceq u$.
Discrete midpoint operators $\bullet$ and $\circ$ on $G$ 
are defined as follows.
For vertices $u, v \in G$, there uniquely exists 
a pair $(a, b)$ of vertices such that $d(u,v) = d(u,a) + d(a, b) + d(b,v)$, 
$d(u,a) = d(b, v)$, and $d(a,b) \leq 1$.
In particular, 
$a$ and $b$ are equal or adjacent, and hence comparable.
Let $u \bullet v := \min \{a,b\}$ 
and $u \circ v := \max \{a,b\}$. 

Consider the product $G^n$ of $G$; see Figure~\ref{fig:treegrid}
for $G^2$.
\begin{figure}[t]
	\centering
		\includegraphics[scale=0.8]{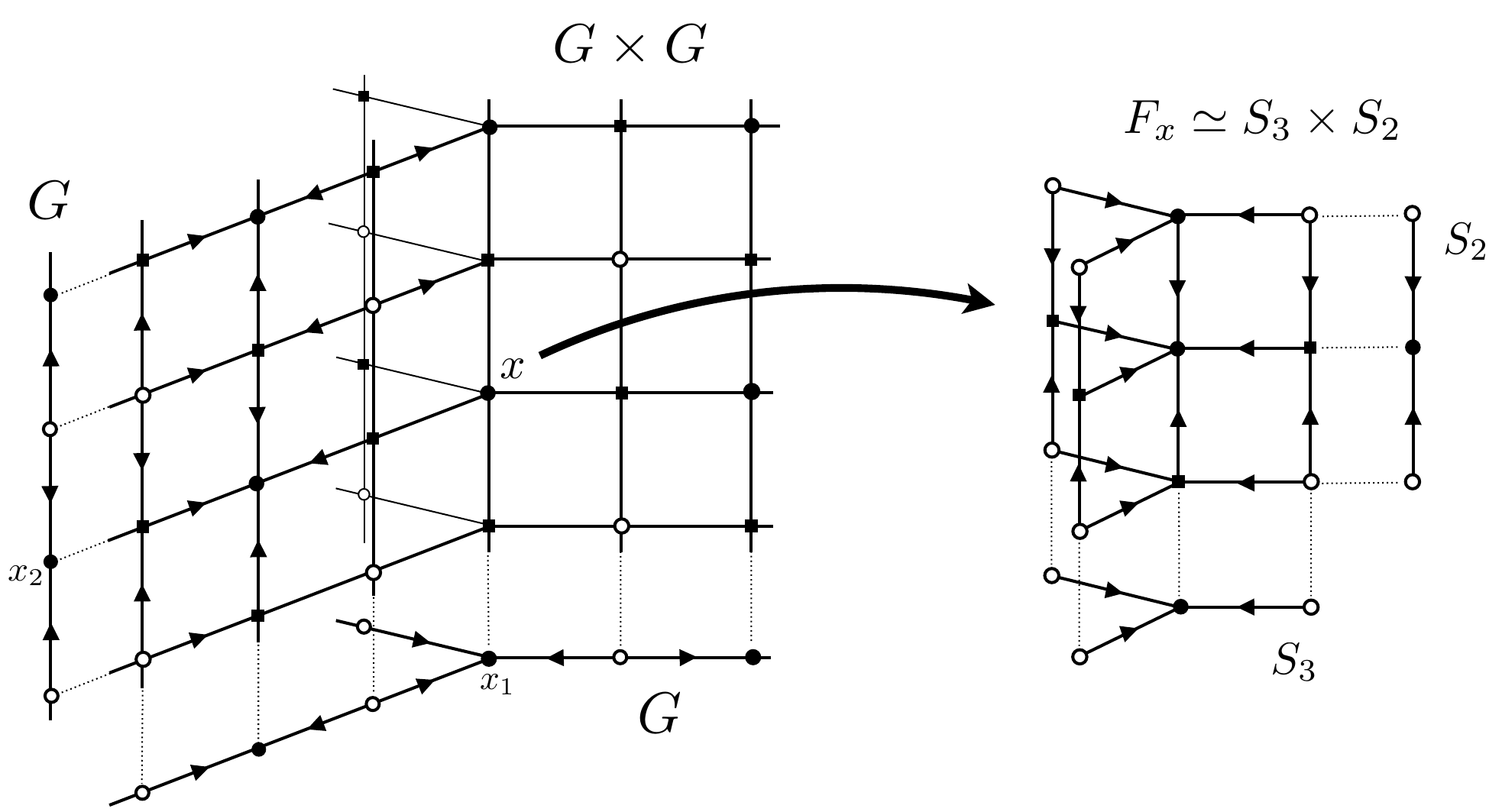}
		\caption{Tree-grid $G^2$. 
			Black and white  points in $G$ represent
			vertices in $B$ and $W$, respectively, 
			The principal filter of a vertex $x \in G^2$ 
			is picked out to the right, 
			which is a semilattice isomorphic to $S_{3} \times S_2$. 
		}
		\label{fig:treegrid}
\end{figure}\noindent
The operators $\bullet$ and $\circ$ are extended on $G^n$ component-wise: For $x =(x_1,x_2,\ldots,x_n), y = (y_1,y_2,\ldots,y_n) \in G^n$, let $x \bullet y$ and $x \circ y$ be defined by
\begin{equation}
x \bullet y :=  (x_1 \bullet y_1, x_2 \bullet y_2,\ldots, x_n \bullet y_n),
\  x \circ y := (x_1 \circ y_1, x_2 \circ y_2,\ldots, x_n \circ y_n).
\end{equation}

Mimicking the discrete midpoint convexity (\ref{eqn:dmc}), 
a function $g: G^n \to \barRR$ is called {\em L-convex} (or {\em alternating L-convex}~\cite{HH14extendable}) if it satisfies
\begin{equation}\label{eqn:midpoint}
g(x) + g(y) \geq g(x \bullet y) + g( x \circ y) \quad (x,y \in G^n).
\end{equation}
By this definition, 
a local-to-global optimality criterion is obtained 
in a straightforward way. 
Recall the notation that $I_x$ and $F_x$ are the principal ideal and filter of vertex $x \in G^n$. 
\begin{Lem}[{\cite{HH14extendable}}]\label{lem:L-optimality1} 
	Let $g: G^n \to \overline{\RR}$ be an L-convex function.
If $x \in \dom g$ is not a minimizer of $g$, there exists $y \in I_x \cup F_x$ with $g(y) < g(x)$.  
\end{Lem}
\begin{proof} Suppose that $x \in \dom g$ is not a minimizer of $g$.
	There is $z \in \dom g$ such that $g(z) < g(x)$. 
	Choose such $z$ with minimum $\max_i d(x_i,z_i)$. 
	By $g(x) + g(z) \geq g(x \bullet z) + g(x \circ z)$,
	it necessarily holds that $\max_i d(x_i,z_i) \leq 1$,
	and one of $x \bullet z \in I_x$  and $x \circ z \in F_x$ 
	has a smaller value of $g$ than $x$.  
\end{proof}
This lemma says that $I_x$ and $F_x$ are ``neighborhoods" of $x$ 
for which the local optimality is defined. 
This motivates us to 
consider the class of functions appearing as the restrictions of $g$ to $I_x$ and $F_x$ for $x \in G^n$, which we call the {\em localizations} of $g$.
The localizations of L-convex functions give rise to
a class of submodular-type discrete convex functions 
known as {\em $k$-submodular functions}~\cite{HK12}.
To explain this fact, we introduce a class of semilattices 
isomorphic to $I_x$ or $F_x$. 

For a nonnegative integer $k$, 
let $S_k$ denote a $(k+1)$-element set with a special element $0$.
Define a partial order $\preceq$ on $S_k$ 
by $0 \preceq u$ for $u \in S_{k} \setminus \{0\}$ with no other relations.
Let $\sqcup$ and $\sqcap$ be binary operations on $S_k$
defined by
\[
u \sqcap v :=
\left\{\begin{array}{ll}
\min\{u,v\} & {\rm if}\ u \preceq v\ {\rm or}\ v \preceq u, \\
0 &{\rm otherwise,} 
\end{array}\right. \ 
 u \sqcup v :=
\left\{\begin{array}{ll}
\max\{u,v\} & {\rm if}\ u \preceq v\ {\rm or}\ v \preceq u, \\
0 &{\rm otherwise.} 
\end{array}\right. 
\]
For an $n$-tuple $\pmb k =(k_1,k_2,\ldots,k_n)$ of nonnegative integers, 
let ${S_{\pmb k}} :=  S_{k_1} \times S_{k_2} \times \cdots \times S_{k_n}$. 
A function $f: {S_{\pmb k}} \to \barRR$ is 
{\em $\pmb k$-submodular} 
if it satisfies
\begin{equation}
f(x) + f(y) \geq f(x \sqcap y) + f(x \sqcup y) \quad (x,y \in {S_{\pmb k}}),
\end{equation}
where operators $\sqcap$ and $\sqcup$ are extended to $S_{{\pmb k}}$ component-wise.
Then ${\pmb k}$-submodular functions are identical 
to submodular functions if $\pmb k= (1,1,\ldots,1)$ and
to bisubmodular functions if $\pmb k= (2,2,\ldots,2)$.

Let us return to tree-product $G^n$.
For every point $x \in G^n$, 
the principal filter $F_{x}$ is isomorphic to 
$S_{\pmb k}$ for ${\pmb k} = (k_1,k_2,\ldots,k_n)$, 
where $k_i = 0$ if $x_i \in W$ and
$k_i$ is equal to the degree of $x_i$ in $G$ if $x_i \in B$.
Similarly for the principal ideal $I_x$ (with partial order reversed).

Observe that operations $\bullet$ and $\circ$ coincides 
with $\sqcap$ and  $\sqcup$ (resp. $\sqcup$ and $\sqcap$) 
in any principle filter (resp. ideal).
Then an L-convex function on $G^n$ is locally $\pmb k$-submodular in the following sense.
\begin{Lem}[\cite{HH14extendable}]\label{lem:local_1}
An L-convex function
$g: G^n \to \overline{\RR}$ is $\pmb k$-submodular on $F_x$ and on $I_x$ for every $x \in \dom g$.	
\end{Lem}
In particular, an L-convex function on a tree-grid can be minimized 
via successive $\pmb k$-submodular function minimizations; 
see Section~\ref{subsec:SDA}.

\subsection{L-convex function on twisted tree-grid}\label{subsec:skew}

Next we consider a variant of a tree-grid, which is
obtained by {\em twisting} the product of two trees, 
and by taking the product.
Let $G$ and $H$ be infinite trees 
without vertices of degree one.
Consider the product $G \times H$, which is also a bipartite graph.
Let $B$ and $W$ denote the color classes of $G \times H$.
For each 4-cycle $C$ of $G \times H$, 
add a new vertex $w_C$ and new four edges joining $w_C$ and vertices in $C$.
Delete all original edges of $G \times H$, 
i.e., all edges non-incident to new vertices.
The resulting graph
is denoted by $G \boxtimes H$.
Endow $G \boxtimes H$ with an edge-orientation such that 
$x \to y$ if and only if $x \in W$ or $y \in B$.
Let $\preceq$ denote the induced partial order on $G \boxtimes H$. 
See Figure~\ref{fig:GXH} for~$G \boxtimes H$.
	\begin{figure}[t]
		\begin{center}
			\includegraphics[scale=0.8]{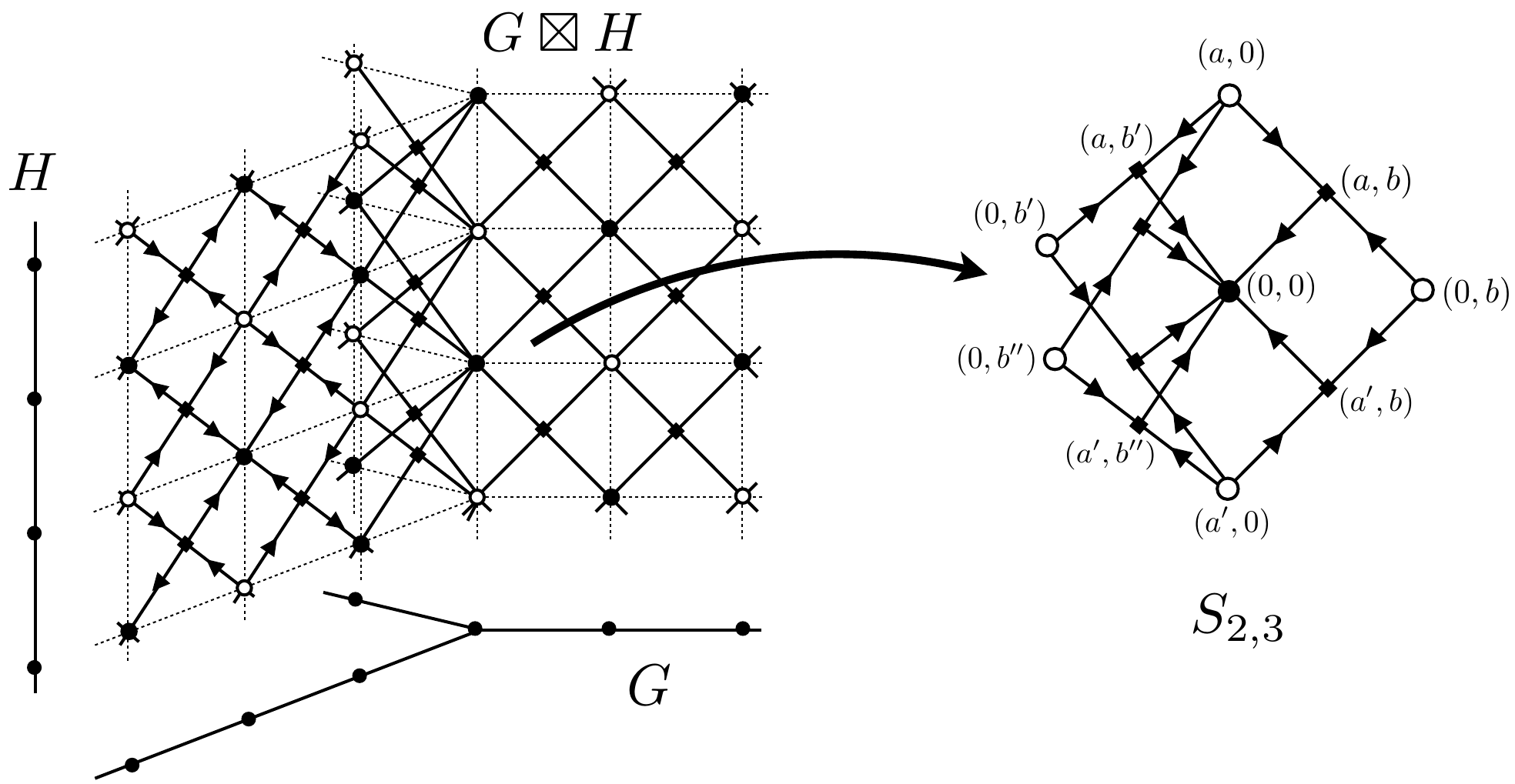}
			\caption{Twisted tree-grid $G \boxtimes H$. 
				Dotted lines represent edges of $G \times H$. 
				Black and white (round) points represent
				vertices in $B$ and $W$, respectively, 
				while square points correspond to 4-cycles in $G \times H$. 
				The principal filter of a black vertex 
				is picked out to the right, which is a semilattice isomorphic to $S_{2,3}$. 
				 }
			\label{fig:GXH}
		\end{center}
	\end{figure}\noindent

Discrete midpoint operators $\bullet$ and $\circ$ are defined as follows.
Consider two vertices $x,y$ in $G \boxtimes H$.
Then $x \in C$ or $x = w_C$ for some 4-cycle $C$ in $G \times H$. 
Similarly, $y \in D$ or $y = w_D$ for some 4-cycle $D$ in $G \times H$.
There are infinite paths $P$ in $G$ and $Q$ in $H$
such that $P \times Q$ contains both $C$ and $D$.
Indeed, consider the projections of $C$ and $D$ to $G$, 
which are edges. Since $G$ is a tree, 
we can choose $P$ as an infinite path in $G$ containing both of them.
Similarly for $Q$.
Now the subgraph $P \boxtimes Q$ of $G \boxtimes H$ coincides with (45-degree rotation of)
the product of two paths with the zigzag orientation 
as in Section~\ref{subsec:tree-grids}.
Thus, in $P \boxtimes Q$, 
the discrete midpoint points $x \bullet y$ and $x \circ y$ are defined. 
They are determined independently of the choice of $P$ and~$Q$.

Consider the $n$-product $(G \boxtimes H)^n$ of $G \boxtimes H$, 
which is called a {\em twisted tree-grid}.
Similarly to the previous case, an {\em L-convex function} 
on $(G \boxtimes H)^n$ is defined as
a function $g: (G \boxtimes H)^n \to \barRR$ satisfying 
\begin{equation}\label{eqn:midpoint2}
g(x) + g(y) \geq g(x \bullet y) + g( x \circ y) \quad (x,y \in (G \boxtimes H)^n),
\end{equation}
where operations $\bullet$ and $\circ$ are extended on $(G \boxtimes H)^n$ component-wise as before.
Again the following holds, where 
the proof is exactly the same as Lemma~\ref{lem:L-optimality1}.
\begin{Lem}\label{lem:L-optimality2} 
	Let $g: (G \boxtimes H)^n \to \overline{\RR}$ be an L-convex function.
	If $x \in \dom g$ is not a minimizer of $g$, there is $y \in I_x \cup F_x$ with $g(y) < g(x)$.  
\end{Lem}

As before,
we study localizations of 
an L-convex function on $(G \boxtimes H)^n$. 
This naturally leads to 
further generalizations of ${\pmb k}$-submodular functions.
For positive integers $k,l$, 
consider product $S_{k} \times S_l$ of 
$S_k$ and $S_l$. 
Define a partial order $\preceq'$ on $S_{k} \times S_l$
by $(0,0) \preceq' (a,b) \preceq' (a,0)$ and $(a,b) \preceq' (0,b)$ for $(a,b) \in S_{k} \times S_l$ with $a \neq 0$ and $b \neq 0$.
The resulting poset is denoted by $S_{k,l}$; 
note that $\preceq'$ is different from the direct product order.

The operations $\sqcup$ and $\sqcap$ on $S_{k,l}$ 
are defined as follows.
For $p,q \in S_{k,l}$, 
there are  
$a,a' \in S_k \setminus \{0\}$ and $b,b' \in S_{l} \setminus \{0\}$ 
such that $p,q \in \{0,a,a'\} \times \{0,b,b'\}$.
The restriction of $S_{k,l}$ to $\{0,a,a'\} \times \{0,b,b'\}$ is isomorphic to 
${S_2} \times S_2 = \{-1,0,1\}^2$ (with order $\preceq$), 
where an isomorphism $\varphi$ is given by
\begin{eqnarray*}
&& (0,0) \mapsto (0,0),\\ 
&& (a,b) \mapsto (1,0), (a',b) \mapsto (0,-1),
(a,b') \mapsto (0,1), (a',b') \mapsto (-1,0), \\
&& (a,0)  \mapsto (1,1), (a',0) \mapsto (-1,-1),
(0,b) \mapsto (1,-1), (0,b') \mapsto (-1,1).
\end{eqnarray*}
Then $p \sqcup q := \varphi^{-1}(\varphi(p) \sqcup \varphi(q))$ and $p \sqcap q := \varphi^{-1}(\varphi(p) \sqcap \varphi(q))$. 
Notice that they are  determined independently of the choice of 
$a,a',b,b'$ and $\varphi$.

For a pair of $n$-tuples $\pmb k = (k_1,k_2,\ldots,k_n)$ 
and $\pmb l = (l_1,l_2,\ldots,l_n)$, 
let $S_{\pmb k, \pmb l} := S_{k_1,l_1} \times S_{k_2,l_2} \times \cdots \times S_{k_n, l_n}$.
A function $f: {S_{\pmb k, \pmb l}} \to \barRR$ is 
{\em $(\pmb k, \pmb l)$-submodular} 
if it satisfies
\begin{equation}
f(x) + f(y) \geq f(x \sqcap y) + f(x \sqcup y) \quad (x,y \in {S_{\pmb k, \pmb l}}).
\end{equation}
Poset $S_{\pmb k, \pmb l}$ contains 
$S_{0,l} \simeq S_l$ and
{\em diamonds} (i.e., a modular lattice of rank 2) 
as $(\sqcap, \sqcup)$-closed sets. 
Thus $(\pmb k, \pmb l)$-submodular functions  can be viewed as a common generalization of  $\pmb k$-submodular functions and
submodular functions on diamonds.

Both the principal ideal and filter of each vertex in $(G \boxtimes H)^n$ 
are isomorphic to $S_{\pmb k, \pmb l}$ for some $\pmb k, \pmb l$, 
in which $\{  \bullet,\circ\}$ are equal to $\{\sqcap, \sqcup\}$.  
Thus we have:
\begin{Lem}\label{lem:local_2}
An L-convex function $g: (G \boxtimes H)^n \to \overline{\RR}$ is 
$(\pmb k, \pmb l)$-submodular on $I_x$ and on $F_x$ for every $x \in \dom g$	
\end{Lem}
\subsection{Steepest descent algorithm}\label{subsec:SDA}
The two classes of L-convex functions  in the previous subsection
can be minimized by the same principle,  
analogous to the steepest descent algorithm for L$^\natural$-convex functions.
Let $\Gamma$ be a tree-product $G^n$ or 
a twisted tree-product $(G \boxtimes H)^n$.
We here consider L-convex functions $g$ 
on $\Gamma$ such that a minimizer of $g$ exists.
\begin{description}
	\item[{\bf Steepest Descent Algorithm (SDA)}]
	\item[Input:] An L-convex function $g: \Gamma  \to \overline\RR$ and an initial point $x \in \dom g$.
	\item[Output:] A minimizer $x$ of $g$.
	\item[Step 1:] Find a minimizer $y$ of $g$ 
	over $I_{x} \cup F_{x}$.
	\item[Step 2:] If $g(x) = g(y)$, then output $x$ and stop; $x$ is a minimizer.
	\item[Step 3:] Let $x := y$, and go to step 1.
\end{description}
The correctness of this algorithm follows immediately from
Lemmas~\ref{lem:L-optimality1} and \ref{lem:L-optimality2}.
A minimizer $y$ in Step 1 is particularly called a {\em steepest direction} at $x$. Notice that a steepest direction 
is obtained by minimizing $g$ over $I_x$ and over $F_x$, 
which are $\pmb k$- or $(\pmb k,\pmb l)$-submodular function minimization by Lemmas~\ref{lem:local_1} and \ref{lem:local_2}.

Besides its conceptual simplicity,  
there remain two issues in applying SDA to specific combinatorial optimization problems:
\begin{itemize}
	\item How to minimize $g$ over $I_x$ and over $F_x$. 
	\item How to estimate the number of iterations.
\end{itemize}
We first discuss the second issue.
In fact, there is a surprisingly simple and sharp iteration bound, 
analogous to the case of 
L$^\natural$-convex functions~\cite{MurotaShioura14}.  
If $\Gamma$ is a tree-grid $G^n$, 
then define $\Gamma^{\Delta}$ as 
the graph obtained from $\Gamma$ by adding an edge
to each pair of (distinct) vertices $x,y$ 
with $d(x_i,y_i) \leq 1$ for $i=1,2,\ldots,n$.
If $\Gamma$ is a twisted tree-grid $(G \boxtimes H)^n$, 
then define $\Gamma^{\Delta}$ as 
the graph obtained from $\Gamma$ by adding an edge to each pair of
(distinct) vertices $x,y$ such that 
$x_i$ and $y_i$ 
belong to a common $4$-cycle in $G \boxtimes H$ for each $i=1,2,\ldots,n$. 
Let $d_{\Delta} := d_{\Gamma^{\Delta}}$.
Observe that $x$ and $y \in I_{x} \cup F_x$ 
are adjacent in $\Gamma^{\Delta}$.
Hence the number of iterations of SDA is at least 
the minimum distance 
$d_{\Delta}(x, \opt(g)) := \min \{ d_{\Delta}(x,y) \mid y \in \opt(g)\}$
from the initial point $x$ 
and the minimizer set ${\rm opt}(g)$ of $g$.
This lower bound is almost tight. 
\begin{Thm}[\cite{HH14extendable,HH15node_multi}]\label{thm:l_inf_bound}
	The number of the iterations of SDA 
	applied to L-convex function $g$ and initial point $x \in \dom g$ 
	is 
	at most $d_{\Delta}(x, {\rm opt}(g)) + 2$.
\end{Thm}
For specific problems (considered in Section~\ref{sec:applications}), 
the upper bound of $d_{\Delta}(x, {\rm opt}(g))$ 
is relatively easier to be estimated.

Thus we concentrate only on the first issue.
Minimizing $g$ over $I_x$ 
and $F_x$ is a $\pmb k$-submodular function minimization if $\Gamma$ is a tree-grid 
and is a $(\pmb k, \pmb l)$-submodular function minimization 
if $\Gamma$ is a twisted tree-grid.  
Currently no polynomial time algorithm 
is known for $\pmb k$-submodular function minimization 
under the oracle model.
%
However, under several important special cases 
(including VCSP model), 
the above submodular functions 
can be minimizable in polynomial time, 
and SDA is implementable; see Section~\ref{sec:oriented_modular}.
Moreover, for further special classes of $\pmb k$- or 
$({\pmb k}, {\pmb l})$-submodular functions
arising from our target problems in Section~\ref{sec:applications}, 
a fast and efficient minimization 
via network or submodular flows is possible, 
and hence the SDA framework
brings efficient combinatorial polynomial time algorithms.

\subsection{L-convex function on Euclidean building}\label{subsec:building}
The arguments in the previous sections 
are naturally generalized to 
the structures known as 
{\em spherical} and {\em Euclidean buildings of type C}; see~\cite{BuildingBook,TitsBuilding} for the theory of buildings. 
We here explain a building-theoretic approach 
to L-convexity and submodularity. 
We start with a local theory; we introduce a {\em polar space} and  
{\em submodular functions} on it. 
Polar spaces are equivalent to spherical buildings of type C~\cite{TitsBuilding}, 
and generalize domains
 $S_{\pmb k}$ and $S_{\pmb k, \pmb l}$ 
for $\pmb  k$-submodular and $(\pmb k, \pmb l)$-submodular functions.

A polar space $L$ of rank $n$ is defined as 
a poset endowed with a system of subposets, called {\em polar frames}, 
satisfying the following axioms:  
\begin{itemize}
	\item[P0:]\ Each polar frame is isomorphic to ${S_2}^n$.
	\item[P1:] 
	\ For two chains $C,D$ in $L$, 
	there is a polar frame $F$ containing them.
	\item[P2:]\ If polar frames $F,F'$ 
	both contain two chains $C,D$, then there is an isomorphism $F \to F'$ being identity on $C$ and $D$.
\end{itemize}
Namely a polar space is viewed as an amalgamation of several ${S_2}^n$.
Observe that $S_{k}$ and $S_{k, l}$ (with $k,l \geq 2$) are polar spaces, 
and so are their products.

Operators $\sqcap$ and $\sqcup$ on polar space $L$ are defined as follows.
For $x,y \in L$, consider a polar frame $F \simeq {S_2}^n$ containing $x,y$ (via P1), and
$x \sqcap y$ and $x \sqcup y$ can be defined in $F$. 
One can derive from the axioms that $x \sqcap y$ and $x \sqcup y$ are determined independently of the choice of a polar frame $F$. 
Thus operators $\sqcap$ and $\sqcup$ are well-defined.

A {\em submodular function} on a polar space $L$ is a function $f:L \to \RR$ satisfying
\begin{equation}\label{eqn:submo_polar1}
	f(x) + f(y) \geq f(x \sqcap y) + f(x \sqcup y)  \quad  (x,y \in L).
\end{equation}
Equivalently, a submodular function on a polar space
is a function being bisubmodular on each polar frame.

Next we introduce a {\em Euclidean building (of type C)} and {\em L-convex functions} on it.
A Euclidean building is simply defined from the above axioms by replacing 
${S_2}^n$ by $\ZZ^n$, 
where $\ZZ$ is zigzagly ordered as
\[
\cdots \succ -2 \prec -1 \succ 0 \prec 1 \succ 2 \prec \cdots.
\]
Namely a Euclidean building of rank $n$ is a poset $\Gamma$ endowed 
with a system of subposets, called {\em apartments}, satisfying:
\begin{itemize}
	\item[B0:] \ \ Each apartment is isomorphic to $\ZZ^n$.
	\item[B1:] \ \ For any two chains $A, B$ in $\Gamma$,
	there is an apartment $\Sigma$ containing them.
	\item[B2:] \ \ If $\Sigma$ and $\Sigma'$ are apartments containing two chains $A, B$,
	then there is an isomorphism $\Sigma \to \Sigma'$ being identity on $A$ and $B$.
\end{itemize}
A tree-product $G^n$ and twisted tree-product $(G \boxtimes H)^n$ are Euclidean buildings (of rank $n$ and $2n$, respectively).
In the latter case,  
apartments are given by 
$(P_1 \boxtimes Q_1) \times (P_2 \boxtimes Q_2) \times \cdots \times (P_n \boxtimes Q_n) \simeq \ZZ^{2n}$ 
for infinite paths $P_i$ in $G$ and $Q_i$ in $H$ 
for $i=1,2,\ldots,n$.

The discrete midpoint operators are defined as follows.
For two vertices $x,y \in \Gamma$, choose an apartment $\Sigma$ containing $x,y$.
Via isomorphism $\Sigma \simeq \ZZ^n$, 
discrete midpoints $x \bullet y$  and $x \circ y$ 
are defined as in the previous section.
Again, $x \bullet y$ and $x \circ y$ and 
independent of the choice of apartments. 

An {\em L-convex function} on $\Gamma$ is a function $g:\Gamma \to \overline{\RR}$ 
satisfying
\begin{equation}\label{eqn:midpoint_building}
	g(x) + g(y) \geq g(x \bullet y) + g( x \circ y) \quad (x,y \in \Gamma).
\end{equation}
Observe that each principal ideal 
and filter of $\Gamma$ are polar spaces. 
Then the previous Lemmas~\ref{lem:L-optimality1}, 
\ref{lem:local_1}, \ref{lem:L-optimality2}, \ref{lem:local_2}, and Theorem~\ref{thm:l_inf_bound} 
are generalized as follows:
\begin{itemize}
	\item An L-convex function $g:\Gamma \to \overline{\RR}$ is submodular on polar spaces $I_x$ and $F_x$ 
	for every $x \in \Gamma$.
	\item If $x$ is not a minimizer of $g$, then there is $y \in I_x \cup F_x$ with $g(y) < g(x)$. In particular, the steepest descent algorithm (SDA) is well-defined, and correctly obtains a minimizer of $g$ (if it exists).
	\item The number of iterations of SDA for $g$ and initial point $x \in \dom g$ 
	is bounded by $d_{\Delta}(x,\opt (g))+2$, where 
	$d_{\Delta}(y,z)$ is the $l_{\infty}$-distance between $y$ and $z$ 
	in the apartment $\Sigma \simeq \ZZ^n$ containing $y,z$.
\end{itemize}

Next we discuss a convexity aspect of L-convex functions.
Analogously to 
$\ZZ^n \hookrightarrow \RR^n$, there is a continuous metric space $K(\Gamma)$ 
into which $\Gamma$ is embedded.
Accordingly, any function $g$ on $\Gamma$ is extended 
to a function $\bar g$ on $K(\Gamma)$, which is 
an analogue of the Lov\'asz extension.
It turns out that the L-convexity of $g$ is equivalent to 
the convexity of $\bar g$ with respect to the metric on $K(\Gamma)$.

We explain this fact more precisely.
Let $K(\Gamma)$ be the set of 
formal sums
\[
\sum_{p \in \Gamma} \lambda(p) p
\]
of vertices in $\Gamma$ 
for $\lambda: \Gamma \to \RR_+$ satisfying 
that $\{ p \in \Gamma \mid \lambda(p) \neq 0 \}$ is a chain 
and $\sum_{p} \lambda(p) = 1$.
For a chain $C$, the subset of form $\sum_{p \in C} \lambda(p) p$ 
is called a {\em simplex}.
For a function $g: \Gamma \to \overline{\RR}$, the {\em Lov\'asz extension} $g: K(\Gamma) \to \overline{\RR}$ is defined by
\begin{equation}
	\bar g(x) := \sum_{p \in \Gamma} \lambda(p) g(p) \quad 
	\left( x = \sum_{p \in \Gamma} \lambda(p) p \in K(\Gamma) \right).
\end{equation}
The space $K(\Gamma)$ is endowed with a natural metric.
For an apartment $\Sigma$, 
the subcomplex $K(\Sigma)$ is isomorphic to 
a simplicial subdivision of $\RR^n$ into simplices 
of vertices 
\[
z, \quad z + s_{i_1}, \quad z + s_{i_1} + s_{i_2},\quad  \ldots,\quad z + s_{i_1} + s_{i_2} + \cdots + s_{i_n}
\]
for all even integer vectors $z$, permutations $(i_1,i_2,\ldots,i_n)$ of $\{1,2,\ldots,n\}$ and $s_{i} \in \{e_i, - e_i \}$ for $i=1,2,\ldots,n$, where 
$e_i$ denotes the $i$-th unit vector.
Therefore one can metrize $K(\Gamma)$ 
so that, for each apartment $\Sigma$, 
the subcomplex $K(\Sigma)$ is an isometric subspace of $K(\Gamma)$ and is isometric to the Euclidean space $(\RR^n, l_2)$.
Figure~\ref{fig:KGXH} illustrates the space $K(\Gamma)$ for twisted tree-grid $\Gamma = G \boxtimes H$.
\begin{figure}[t]
	\begin{center}
		\includegraphics[scale=0.8]{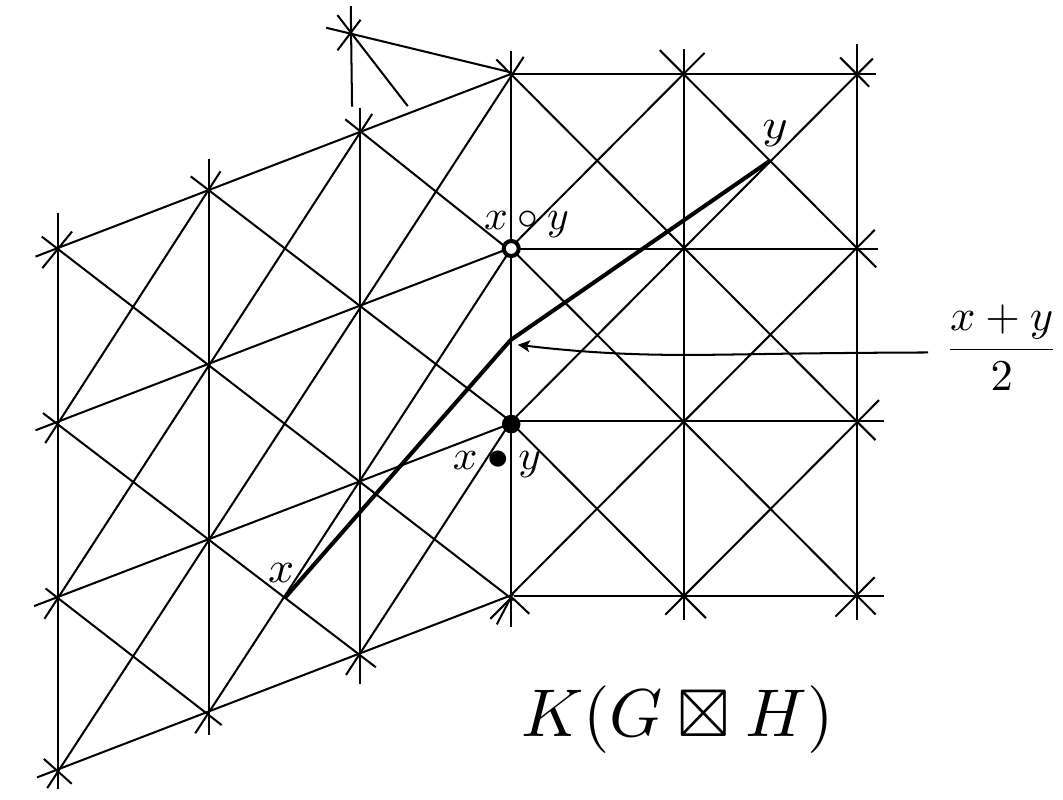}
		\caption{The space $K(G \boxtimes H)$, which is constructed 
			from $G \boxtimes H$
			by filling Euclidean right triangle to each chain of length two.
			The Lov{\'a}sz extension is 
			a piecewise interpolation with respect to this triangulation.
			Each apartment is isometric to the Euclidean plane, 
			in which geodesics are line segments.
			The midpoint $(x+y)/2$ of vertices $x,y$ lies on 
			an edge ($1$-dimensional simplex).
			Discrete midpoints $x \circ y$ and $x \bullet y$
			are obtained by rounding $(x+y)/2$ to the ends of the edge
			with respect to the partial order.
		}
		\label{fig:KGXH}
	\end{center}
\end{figure}\noindent
This metric space $K(\Gamma)$ is known as 
the {\em standard geometric realization} of $\Gamma$; see \cite[Chapter 11]{BuildingBook}.
It is known that $K(\Gamma)$ is a {\em CAT(0) space} (see \cite{BrHa}), 
and hence uniquely geodesic, i.e., every pair of points can be joined 
by a unique geodesic (shortest path).
The unique geodesic for two points $x,y$ is given as follows.
Consider an apartment $\Sigma$ with $x,y \in K(\Sigma)$, and
identify $K(\Sigma)$ with $\RR^n$ and $x,y$ with points in $\RR^n$.
Then the line segment $\{ \lambda x + (1 - \lambda) y \mid 0 \leq \lambda \leq 1\}$
in $\RR^n \simeq K(\Sigma) \subseteq K(\Gamma)$ 
is the geodesic between $x$ and $y$.  
In particular, a convex function on $K(\Gamma)$ 
is defined as a function $g:K(\Gamma) \to \overline{\RR}$ satisfying 
\begin{equation}
	\lambda g(x) + (1 - \lambda) g(y) \geq g(\lambda x + (1- \lambda)y) 
	\quad (x,y \in K(\Gamma), 0 \leq \lambda \leq 1),
\end{equation} 
where point $\lambda x + (1- \lambda)y$ is considered in the apartment $\RR^n \simeq K(\Sigma)$.
Notice that the above $( x \bullet y, x \circ y)$ 
is equal to the unique pair $(z,z')$ with the property that 
$z \preceq z'$ and $(x + y)/2 = (z+ z')/2$ 
(in $\RR^n \simeq K(\Sigma)$).
In this setting, the L-convexity is characterized as follows.
\begin{Thm}[\cite{HH16L-convex}]\label{thm:convex}
	Let $\Gamma$ be a Euclidean building of type C.
	For $g: \Gamma \to \overline\RR$, the following conditions are equivalent:
	\begin{itemize}
		\item[{\rm (1)}] \ $g$ is L-convex.
		\item[{\rm (2)}] \ The Lov\'asz extension $\bar g$ of $g$ is convex.
		\item[{\rm (3)}] 
		\ $g$ is submodular on $I_x$ and on $F_x$ for every $x \in \dom g$, and $\dom g$ is chain-connected.
	\end{itemize}
\end{Thm}
Here a subset $X$ of vertices in $\Gamma$ is {\em chain-connected} if 
every $p,q \in X$ there is a sequence $p = p_0,p_1,\ldots,p_k = q$ 
such that $p_i \preceq p_{i+1}$ or $p_i \succeq p_{i+1}$.
\begin{Rem}{\rm 
	A Euclidean building $\Gamma$ consisting of 
	a single apartment is identified with $\ZZ^n$, and 
	$K(\Gamma)$ is a simplicial subdivision of $\RR^n$.
	Then L-convex functions on $\Gamma$ as defined here coincide with 
	 {\em UJ-convex functions} by Fujishige~\cite{Fujishige14}, 
	where he defined them via  
	the convexity property (2) in Theorem~\ref{thm:convex}.}
\end{Rem}

\section{Application}\label{sec:applications}
In this section, we demonstrate SDA-based algorithm design for two important multiflow problems 
from which L-convex functions above actually arise.
The first one is the minimum-cost multiflow problem, and the second is the node-capacitated multiflow problem.
Both problems also arise as (the dual of) LP-relaxations of 
other important network optimization problems, and are connected to good approximation algorithms.
The SDA framework brings fast combinatorial polynomial time algorithms 
to both problems for which such algorithms had not been known before. 

We will use standard notation for directed and undirected networks. 
For a graph $(V,E)$ and a node subset $X \subseteq V$,  
let $\delta(X)$ denote the set of edges $ij$ with $|\{i,j\} \cap X| = 1$.
In the case where $G$ is directed,
let $\delta^+(X)$ and $\delta^-(X)$ denote 
the sets of edges leaving $X$ and entering $X$, respectively.
Here $\delta(\{i\})$, $\delta^+(\{i\})$ and $\delta^-(\{i\})$ are 
simply denoted by $\delta(i)$, $\delta^+(i)$ and $\delta^-(i)$, respectively.
For a function $h$ on a set $V$ and a subset $X \subseteq V$, 
let $h(X)$ denote $\sum_{i \in X} h(i)$. 

\subsection{Minimum-cost multiflow and terminal backup}\label{sub:backup}

\subsubsection{Problem formulation}
An {\em undirected network} $N = (V, E,c, S)$ consists of an undirected graph $(V, E)$, 
an edge-capacity $c:E \to \ZZ_+$, and a specified set $S \subseteq V$ of nodes, called {\em terminals}.
Let $n := |V|$, $m := |E|$, and $k := |S|$.
An {\em $S$-path} is a path connecting distinct terminals in $S$.
A {\em multiflow} is a pair $({\cal P},f)$ of  a set ${\cal P}$ of $S$-paths
and a flow-value function $f:{\cal P} \to \RR_+$ 
satisfying the capacity constraint:
\begin{equation}
f(e) := \sum \{  f(P) \mid P \in {\cal P}: \mbox{$P$ contains $e$} \} \leq c(e) \quad (e \in E).
\end{equation}
A multiflow $({\cal P}, f)$ is simply written as $f$.
The {\em support} of a multiflow $f$ is the edge-weight 
defined by $e \mapsto f(e)$. 
Suppose further that  the network $N$ is given 
 a nonnegative edge-cost $a:E \to \ZZ_+$ and a node-demand $r:S \to \ZZ_+$ on the terminal set $S$.
The {\em cost} $a(f)$ of a multiflow $f$ is defined as 
$\sum_{e \in E} a(e)f(e)$.
A multiflow $f$ is said to be {\em $r$-feasible} 
if it satisfies 
\begin{equation}
\sum \{ f(P) \mid \mbox{$P \in {\cal P}$: $P$ connects $s$}\} \geq r(s) \quad (s \in S).
\end{equation}
Namely each terminal $s$ 
is connected to other terminals in at least $r(s)$ flows.
The {\em minimum-cost node-demand multiflow problem 
(MCMF)} asks to find 
an $r$-feasible multiflow of minimum cost.

This problem was 
recently introduced by Fukunaga~\cite{Fukunaga14} 
as an LP-relaxation of the following network design problem. 
An edge-weight $u: E \to \RR_+$ is said to be {\em $r$-feasible} 
if $0 \leq u \leq c$ and 
the network $(V,E,u,S)$
has an $(s, S \setminus\{s\})$-flow 
of value at least $r(s)$ for each $s \in S$.
The latter condition is represented as the following cut-covering constraint:
\begin{equation}\label{eqn:cutcover}
u(\delta(X)) \geq r(s) \quad (s \in S, X \subseteq V: X \cap S = \{s\}).
\end{equation}
The (capacitated) {\em terminal backup problem} (TB) 
asks to find 
an integer-valued $r$-feasible edge-weight $u:E \to \ZZ_+$ of minimum-cost $\sum_{e \in E} a(e) u(e)$.
This problem, introduced by 
Anshelevich and Karagiozova~\cite{AK11}, 
was shown to be polynomially solvable~\cite{AK11,BKM15} 
if there is no capacity bound.
The complexity of TB for the general capacitated case is not known.
The natural fractional relaxation,  called 
the {\em fractional terminal backup problem (FTB)}, 
is obtained by relaxing $u: E \to \ZZ_+$ to $u : E \to \RR_+$.
In fact, MCMF and FTB are equivalent in the following sense:
\begin{Lem} [{\cite{Fukunaga14}}] \label{lem:NvsL}
	\begin{itemize}
		\item[{\rm (1)}] \ For an optimal solution $f$ of MCMF, 
		the support of $f$ is an optimal solution of FTB.
		\item[{\rm (2)}] \ For an optimal solution $u$ of FTB, 
		an $r$-feasible multiflow $f$ in $(V,E,u,S)$ exists, 
		and is optimal to MCMF.
	\end{itemize}
\end{Lem}
Moreover, half-integrality property holds:
\begin{Thm}[{\cite{Fukunaga14}}]\label{thm:half-integ}
	There exist half-integral optimal solutions in FTB, MCMF, and their LP-dual.
\end{Thm}	
By utilizing this half-integrality,
Fukunaga~\cite{Fukunaga14} developed a $4/3$-approximation algorithm.
His algorithm, however, uses the ellipsoid method to obtain 
a half-integral (extreme) optimal solution in FTB.

Based on the SDA framework of an L-convex function on a tree-grid,
the paper~\cite{HH14extendable} developed 
a combinatorial weakly polynomial time algorithm for MCMF together with  a combinatorial implementation 
of the $4/3$-approximation algorithm for TB. 
\begin{Thm}[\cite{HH14extendable}]\label{thm:algo1}
	A half-integral optimal solution in MCMF and 
	a $4/3$-approximate solution in TB
	can be obtained in
	$O(n \log (n AC)\, {\rm MF}(kn, km))$ time. 
\end{Thm}
Here ${\rm MF}(n', m')$ denote the time complexity of solving the maximum flow problem on a network of $n'$ nodes and $m'$ edges, and $A := \max \{ a(e) \mid e \in E\}$ 
and $C:= \max \{c(e) \mid e \in E\}$.

It should be noted that MCMF generalizes
the {\em minimum-cost maximum free multiflow problem} considered by Karzanov~\cite{Kar79, Kar94}.
To this problem,  Goldberg and Karzanov~\cite{GK97}
developed a combinatorial weakly polynomial time algorithm.
However the analysis of their algorithm is not easy, 
and the explicit polynomial running time is not given. 
The algorithm in Theorem~\ref{thm:algo1} is the first combinatorial weakly polynomial time algorithm having an explicit running time.

\subsubsection{Algorithm}
Here we outline the algorithm in Theorem~\ref{thm:algo1}.
Let $N = (V, E, c, S)$ be a network, 
$a$ an edge cost, and $r$ a demand. 
For technical simplicity, 
the cost $a$ is assumed to be positive integer-valued $\geq 1$.
First we show that 
the dual of MCMF can be formulated as an optimization over 
the product of {\em subdivided stars}.
For $s \in S$,
let $G_s$ be a path with infinite length and an end vertex $v_s$ of degree one.
Consider the disjoint union  $\bigcup_{s \in S} G_s$ 
and identify all $v_s$ to one vertex $O$.
The resulting graph, called a {\em subdivided star}, 
is denoted by $G$, and the edge-length is defined as $1/2$ uniformly.
Let $d = d_G$ denote the shortest path metric of $G$ 
with respect to this edge-length.

Suppose that $V = \{1,2,\ldots,n\}$.
A {\em potential} is a vertex $p = (p_1,p_2,\ldots,p_n)$ in $G^n$
such that $p_s \in G_s$ for each $s \in S$. 
\begin{Prop}[\cite{HH14extendable}]\label{prop:duality1}
	The minimum cost of an $r$-feasible multiflow is equal to the maximum of
	\begin{equation}\label{eqn:formula1}
	\sum_{s \in S} r(s) d(p_s, O) - \sum_{ij \in E} c(ij) \max \{ d(p_i,p_j) - a(ij), 0 \} 
	\end{equation}
	over all potentials $p = (p_1,p_2,\ldots,p_n) \in G^n$.
\end{Prop}
\begin{proof}[Sketch]
The LP-dual of FTB is given by:
\begin{eqnarray*}
\mbox{Max.} &&  \sum_{s \in S} r(s)\sum_{X: X \cap S = \{s\}} \pi_X 
- \sum_{e \in E} c(e) \max \{ 0, \sum_{X:   e \in \delta (X)} \pi_X  - a(e) \}\\
\mbox{s.t.} && \pi_X \geq 0\ (X \subseteq V: |X \cap S|=1). \nonumber
\end{eqnarray*}
By the standard uncrossing argument, one can show that there always exists an optimal solution $\pi_X$ 
such that  $\{ X  \mid \pi_X > 0 \}$ is {\em laminar}, 
i.e., if $\pi_X,\pi_Y > 0$ 
it holds that $X \subseteq Y$, $Y \subseteq X$, or $X \cap Y = \emptyset$.
Consider the tree-representation of the laminar family $\{ X  \mid \pi_X > 0 \}$. 
Since each $X$ contains exactly one terminal, 
the corresponding tree is necessarily a subdivided star $\tilde G$ with center $O$ and non-uniform edge-length.
In this representation,
each $X$ with $\pi_X > 0$ corresponds to an edge $e_X$ of $\tilde G$, and each node $i$ is associated with a vertex $p_i$ in $\tilde G$.
The length of each edge $e_X$ is defined as $\pi_X$,  and the resulting shortest path metric is denoted by $D$.
Then it holds that 
$\sum_{X: X \cap S = \{s\}} \pi_X = D(p_s, O)$ 
and $\sum_{X:   ij \in \delta (X)} \pi_X  = D(p_i, p_j)$.
By the half-integrality (Theorem~\ref{thm:half-integ}), we can assume that each $\pi_X$ is a half-integer.
Thus we can subdivide $\tilde G$ to $G$ so that 
each edge-length of $G$ is $1/2$, and obtain the formulation in (\ref{eqn:formula1}). 
\end{proof}
Motivated by this fact,  
define $\omega: G^n \to \overline\RR$ by
\begin{equation*}
p \mapsto  - \sum_{s \in S} r(s) d(p_s, O)
+ \sum_{ij \in E} c(ij) 
\max \{ d(p_i,p_j) - a(ij), 0 \} + I(p),
\end{equation*}
where $I$ is the indicator function of the set of all potentials, 
i.e., $I(p) := 0$ if $p$ is a potential and $\infty$ otherwise.
The color classes of $G$ are denoted by $B$ and $W$ with $O \in B$, and $G$ is oriented zigzagly.
In this setting, the objective function $\omega$ of the dual of MCMF 
is an L-convex function on tree-grid $G^n$: 
\begin{Prop}[\cite{HH14extendable}]
	The function $\omega$ is L-convex on $G^n$.
\end{Prop}	

In the following, we show that the steepest descent algorithm for $\omega$ is efficiently implementable with a maximum flow algorithm.
An outline is as follows:
\begin{itemize}
	\item The optimality of a potential $p$ is equivalent to the feasibility of a circulation problem on directed network $D_{p}$ associated with $p$ (Lemma~\ref{lem:optimality1}), 
	where a half-integral optimal multiflow is recovered from an integral circulation (Lemma~\ref{lem:MF}). 
	\item If $p$ is not optimal, then a certificate ($=$ violating cut) of the infeasibility yields a steepest direction $p'$ at $p$ (Lemma~\ref{lem:steepest1}).   
\end{itemize}
This algorithm may be viewed as 
a multiflow version of the dual algorithm~\cite{Hassin83} on 
minimum-cost flow problem; 
see also \cite{Shioura17L-convex_survey} 
for a DCA interpretation of the dual algorithm.

Let $p \in G^n$ be a potential
and $f:{\cal P} \to \RR_+$ an $r$-feasible multiflow, 
where we can assume that $f$ is positive-valued.
Considering $\sum_{e \in E} a(e) f(e) - (- \omega(p)) (\geq 0)$, 
we obtain the complementary slackness condition:  
Both $p$ and $f$ are optimal if and only if
\begin{eqnarray}
f(e) = 0 \hspace{1.1cm}  && (e = ij \in E: d(p_i, p_j) < a(ij)), \label{eqn:slackness1}\\
f(e) = c(e)\hspace{0.69cm}  && (e = ij \in E: d(p_i, p_j) > a(ij)), \label{eqn:slackness2} \\
\sum_{P \in {\cal P}} \{f(P) \mid \mbox{$P$ connects $s$} \} = r(s)
\hspace{0.7cm} 
&& (s \in S: p_s \neq O), \label{eqn:slackness3}\\
 \sum_{k=1,\ldots, \ell} d(p_{i_{k-1}}, p_{i_{k}}) = d(p_s,p_t) &&(P = (s = i_0,i_1,\ldots, i_{\ell} = t) \in {\cal P}).\label{eqn:slackness4}
\end{eqnarray}
The first three conditions are essentially the same as
the kilter condition in the (standard) minimum cost flow problem.
The fourth one is characteristic of multiflow problems, 
which says that an optimal multiflow $f$ induces 
a collection of geodesics in $G$ via embedding 
$i \mapsto p_i$ for an optimal potential $p$.

Observe that these conditions, except the fourth one, are imposed 
on the support of $f$ rather than multiflow $f$ itself.
The fourth condition can also be represented by a support condition  
on an extended (bidirected) network which we construct below; see Figure~\ref{fig:Dp}. 
An optimal multiflow will be recovered from a fractional bidirected flow on this network.
\begin{figure}[t]
	\begin{center}
		\includegraphics[scale=0.8]{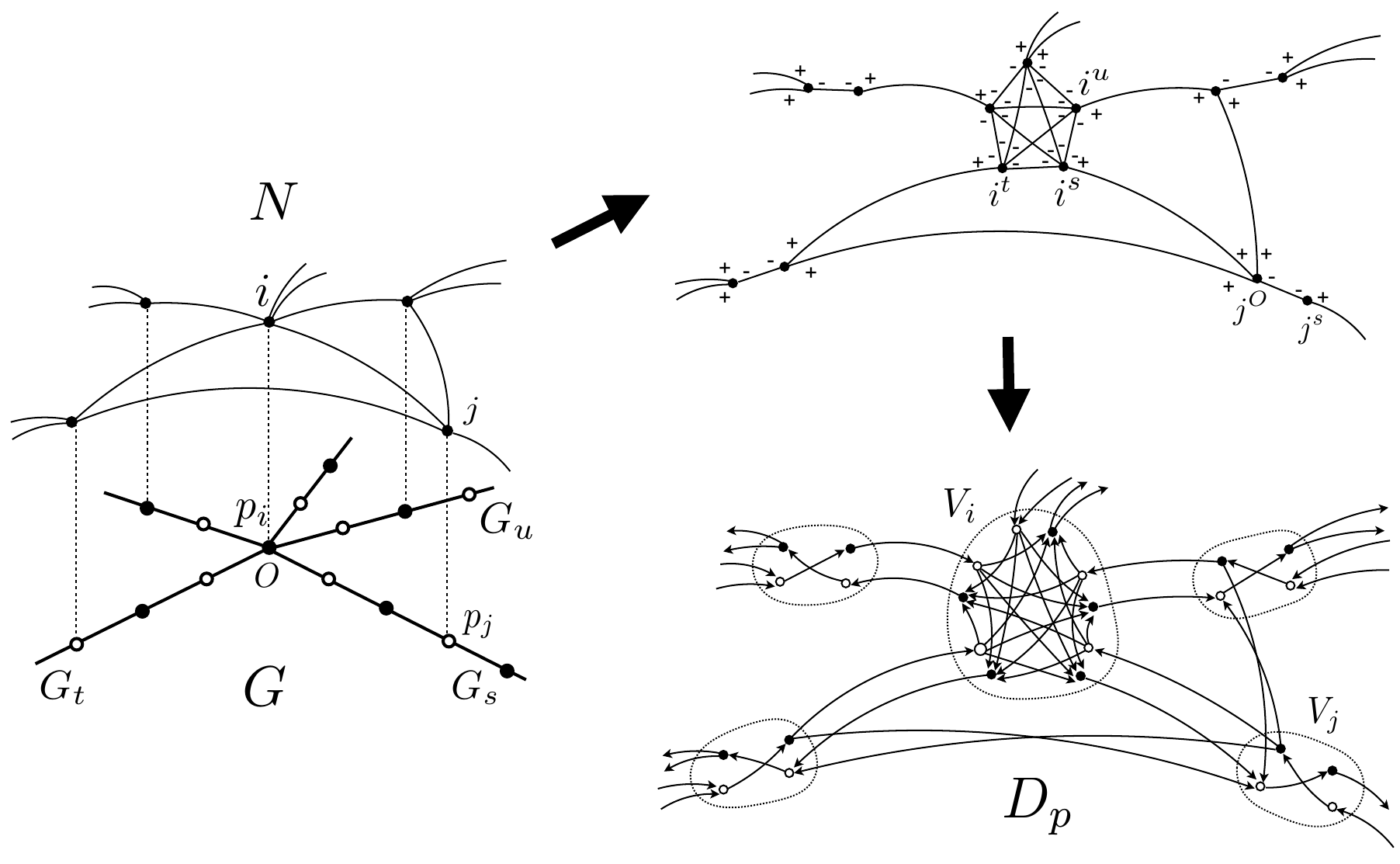}
		\caption{Construction of extended networks. 
			Each node $i$ is split to a clique of size $|S|$ 
			$(p_i=O)$ or size two $(p_i \neq O)$, where
			the edges in these cliques are bidirected edges of sign $(-,-)$. 
			In the resulting bidirected network, 
			drawn in the upper right, 
			the path-decomposition of a bidirected flow 
			gives rise to a multiflow flowing in $G$ geodesically.
			The network $D_p$, drawn in the lower right, 
			is a directed network equivalent to the bidirected network.
			}
		\label{fig:Dp}
	\end{center}
\end{figure}\noindent

Let $E_{=}$ and $E_{>}$ denote the sets of edges $ij$ with $d(p_i,p_j) = a(ij)$ and with $d(p_i,p_j) > a(ij)$, respectively.
Remove all other edges (by (\ref{eqn:slackness1})).
For each nonterminal node $i$ with $p_i = O$, 
replace $i$ by $|S|$ nodes 
$i^s$ $(s \in S)$ and add new edges $i^si^{t}$ 
for distinct $s,t \in S$.
Other node $i$ is  a terminal $s$ or a nonterminal node 
with $p_i \in G_s \setminus \{O\}$. 
Replace each such node $i$ by two nodes $i^s$ and $i^O$, 
and add new edge $i^s i^O$. 
The node $i^s$ for terminal $i=s$ is also denoted by $s$.
The set of added edges is denoted by $E_{-}$.
For each edge $e = ij \in E_{=} \cup E_{>}$, 
replace $ij$ by $i^O j^s$
if $p_i,p_j \in G_s$ and $d(p_i,O) > d(p_j,O)$, and 
replace $ij$ by $i^Oj^O$ if $p_i \in G_s \setminus \{O\}$ and 
$p_j \in G_t \setminus \{O\}$ for distinct $s,t \in S$. 

An edge-weight $\psi: E_{=} \cup E_{>} \cup E_{-} \to \RR$ is called  a {\em $p$-feasible support} if  
\begin{eqnarray}
 0 \leq \psi(e) \leq c(e) && (e \in E_{=}),\label{eqn:p-feasible_cap1} \\
 \psi(e) = c(e)  && (e \in E_{>}), \label{eqn:p-feasible_cap2} \\
\psi (e) \leq 0 \hspace{0.4cm} && (e \in E_{-}), \label{eqn:p-feasible_auxial} \\
 - \psi(\delta(s)) \geq r(s) && (s \in S, p_s = O), \label{eqn:p-feasible_r1}\\
 - \psi(\delta(s)) = r(s) && (s \in S, p_s \neq O), \label{eqn:p-feasible_r2}\\
\psi(\delta({u})) = 0 \hspace{0.4cm} && 
(\mbox{each nonterminal node $u$}).\label{eqn:p-feasible_conser}
\end{eqnarray}
One can see from an alternating-path argument 
that any $p$-feasible support $\psi$
is represented as a weighted sum $\sum_{P \in {\cal P}} f(P) \chi^P$ 
for a set ${\cal P}$ of $S$-paths with nonnegative coefficients $f:{\cal P} \to \RR_+$, 
where each $P$ is a path 
alternately using edges in $E_{-}$ and edges in $E_{=} \cup E_{>}$,
and $\chi^P$ is defined by 
$\chi^P(e) := - 1$ for edge $e$ in $P$ with $e \in E_-$,  
$\chi^P(e) := 1$ for other edge $e$ in $P$, and zero 
for edges not in $P$.
Contracting all edges in $E_{-}$ 
for all paths $P \in {\cal P}$, 
we obtain an $r$-feasible multiflow 
$f^{\psi}$ in $N$, where 
$f^{\psi}(e) \leq c(e)$ and (\ref{eqn:slackness2}) are guaranteed by (\ref{eqn:p-feasible_cap1}) and (\ref{eqn:p-feasible_cap2}),  
and the $r$-feasibility and (\ref{eqn:slackness3})
are guaranteed by (\ref{eqn:p-feasible_r1}) and (\ref{eqn:p-feasible_r2}).
Also, by construction, each path in ${\cal P}$ 
induces a local geodesic in $G$ by $i \mapsto p_i$, 
which must be a global geodesic since $G$ is a tree. This implies (\ref{eqn:slackness4}), 

\begin{Lem}[\cite{HH14extendable}]\label{lem:optimality1}
	\begin{itemize}
		\item[{\rm (1)}] \ A potential $p$ is optimal if and only 
		if a $p$-feasible support $\psi$ exists.
		\item[{\rm (2)}] \ For any $p$-feasible support 
		$\psi$, 
		the multiflow $f^{\psi}$ 
		is optimal to MCMF.
	\end{itemize}
\end{Lem}
Thus, by solving inequalities 
(\ref{eqn:p-feasible_cap1})-(\ref{eqn:p-feasible_conser}), 
we obtain an optimal multiflow 
or know the nonoptimality of $p$.
Observe that this problem is a fractional bidirected flow problem, 
and reduces to the following circulation problem.
Replace each node $u$ by two nodes $u^+$ and $u^-$.
Replace each edge $e = uv \in E_{-}$ 
by two directed edges $e^+ = u^+v^-$ 
and $e^- = v^+u^-$ with lower capacity $\underline{c}(e^+) = \underline{c}(e^-) := 0$
and upper capacity  $\overline{c}(e^+) = \overline{c}(e^-) := \infty$.
Replace each edge $e = uv \in E_{=} \cup E_{>}$ 
by two directed edges $e^+ = u^-v^+$ and $e^- = v^-u^+$, 
where $\overline{c}(e^+) = \overline{c}(e^-) := c(e)$, 
and $\underline{c}(e^+) = \underline{c}(e^-) := 0$ if 
$e \in E_{=}$ and $\underline{c}(e^+) = \underline{c}(e^-) := c(e)$ if 
$e \in E_{>}$.
For each terminal $s \in S$, add edge 
${s^-}s^{+}$, where 
$\underline{c}(s^-s^+) := r(s)$ 
and $\overline{c}(s^-s^+) := \infty$ if $p_s = O$ and 
$\underline{c}(s^-s^+) = \overline{c}(s^-s^+) := r(s)$ if $p_s \neq O$.
Let $D_p$ denote the resulting network, 
which is a variant of 
the {\em double covering network} in the minimum cost multiflow problem~\cite{Kar79,Kar94}.

A {\em circulation} is an edge-weight $\varphi$ on this network $D_p$ satisfying
$\underline{c}(e) \leq \varphi(e) \leq \overline{c}(e)$ for each edge $e$, 
and $\varphi(\delta^+(u)) = \varphi(\delta^-(u))$ for each node $u$. 
From a circulation $\varphi$ in $D_p$,
a $p$-feasible support $\psi$ is obtained by  
\begin{equation}\label{eqn:psi}
\psi(e) := (\varphi(e^+) + \varphi(e^-))/2 \quad (e \in E_{=} \cup E_{>} \cup E_{-}).
\end{equation}
It is well-known that a circulation, if it exists, 
is obtained by solving one maximum flow problem. Thus we have:
\begin{Lem}[\cite{HH14extendable}] \label{lem:MF}
	From an optimal potential, a half-integral optimal multiflow is obtained in $O({\rm MF}(kn,m + k^2 n)$ time.
\end{Lem}	 

Next we analyze the case where a circulation does not exist.
By Hoffman's circulation theorem, 
a circulation exists in $D_p$ if and only if
\[
\kappa(X) := \underline{c}(\delta^- (X)) - \overline{c}(\delta^+ (X)) 
\]
is nonpositive for every node subset $X$.
A node subset $X$ is said to be a {\em violating cut} 
if $\kappa(X)$ is positive, and is said to be {\em maximum} 
if $\kappa(X)$ is maximum among all node-subsets.

From a maximum violating cut, a steepest direction of $\omega$ at $p$ is obtained as follows.
For (original) node $i \in V$, 
let $V_i$ denote the set of nodes in $D_p$ 
replacing $i$ in this reduction process; see Figure~\ref{fig:Dp}.  
Let $V_i^+$ and $V_i^-$ denote the sets 
of nodes in $V_i$ having $+$ label and $-$ label, respectively. 
A node subset $X$ is said to be {\em movable} 
if $X \cap V_i = \emptyset$ or 
$\{ u^{+}\} \cup V_i^- \setminus \{u^-\}$ for some $u^+ \in V_i^+$.
For a movable cut $X$, 
the potential $p^X$ is defined by
\begin{equation}
(p^X)_i = \left\{
\begin{array}{ll}
\mbox{the neighbor of $p_i$ closer to $O$}
& {\rm if}\ X \cap V_i^+ = \{i^{O+}\}, \\
\mbox{the neighbor of $p_i$ in $G_s$ away from $O$} & {\rm if}\ X \cap V_i^+ = \{i^{s+}\}, \\
p_i & {\rm if}\ X \cap V_i = \emptyset.
\end{array}\right.
\end{equation}
\begin{figure}[t]
	\begin{center}
		\includegraphics[scale=1.0]{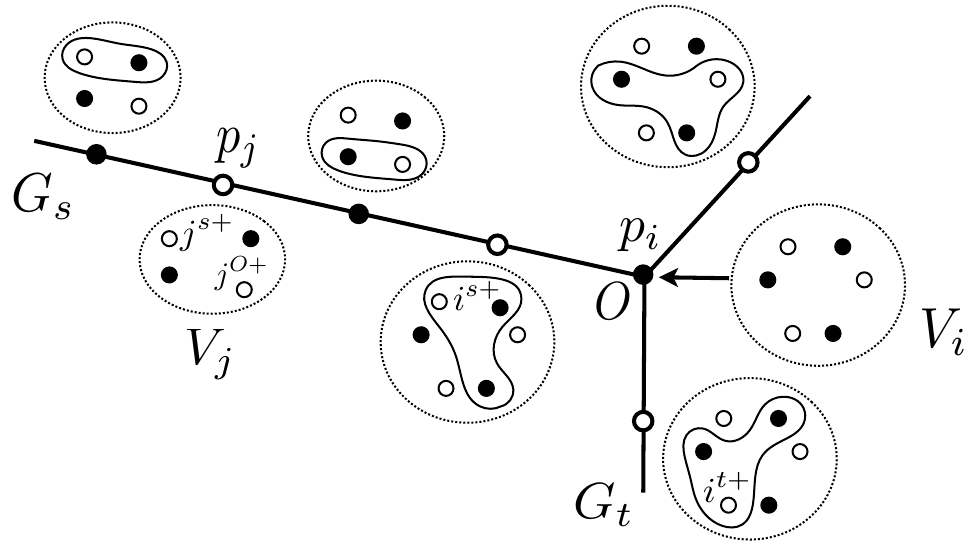}
		\caption{The correspondence 
			between movable cuts in $D_p$ and neighbors of $p$.
			There is a one-to-one correspondence 
			between $I_{p_i} \cup F_{p_i}$ 
			and $\{ X \cap V_i  \mid \mbox{$X$: movable cut}\}$, 
			where $X \cap V_i$ is surrounded by a closed curve, and
			nodes with $+$ label  and $-$ label are represented by white and black points, respectively. 
			By a movable cut $X$, a potential $p$ can be 
			moved to another potential $p^X$ 
			for which $(p^X)_i \in $
			$I_{p_i} \cup F_{p_i}$ $(i \in V)$. 
		}
		\label{fig:cut1}
	\end{center}
\end{figure}\noindent
See Figure~\ref{fig:cut1} for an intuition of $p^X$.
Let $V_I := \bigcup_{i: p_i \in W} V_i$ and $V_F := \bigcup_{i: p_i \in B} V_i$.
Since $a$ is integer-valued, edges between $V_I$ and $V_F$ have the same upper and lower capacity. This implies
$\kappa(X) = \kappa(X \cap V_I) + \kappa(X \cap V_F)$.
Thus, if $X$ is violating, then $X \cap V_I$ or $X \cap V_F$ is violating, 
and actually gives a steepest direction as follows.
\begin{Lem}[\cite{HH14extendable}]\label{lem:steepest1} Let $p$ be a nonoptimal potential.
For a minimal maximum violating cut $X$, both $X \cap V_I$ and $X \cap V_F$ are movable. 
Moreover, $p^{X \cap V_I}$ is a minimizer of $\omega$ over $I_p$ and 
		 $p^{X \cap V_F}$ is a minimizer of $\omega$ over $F_p$. 
\end{Lem}
Now SDA is specialized to MCMF as follows.
\begin{description}
	\item[{\bf Steepest Descent Algorithm for MCMF}]
	\item[Step 0:] Let $p := (O,O,\ldots,O)$.
	\item[Step 1:] Construct network $D_p$.
	\item[Step 2:] If a circulation $\varphi$ exists in $D_p$, 
	then obtain a $p$-feasible support $\psi$ by (\ref{eqn:psi}), 
	and an optimal multiflow $f^{\psi}$ via the path decomposition; stop.
	\item[Step 3:] For a minimal maximum violating cut $X$, 
	choose $p' \in \{p^{X \cap V_I}, p^{X \cap V_F}\}$ with $\omega(p') 
	= \min \{ \omega(p^{X \cap V_I}), \omega(p^{X \cap V_F})\}$, 
	let $p := p'$, and go to step 1.
\end{description}

\begin{Prop}[\cite{HH14extendable}]\label{prop:runin}
The above algorithm runs in $O(n A \cdot {\rm MF}(kn, m+ k^2 n))$ time.	
\end{Prop}
\begin{proof}[Sketch]
	By Theorem~\ref{thm:l_inf_bound}, 
	it suffices to show $\max_i d(O, p_i) = O(n A)$ for some optimal potential $p$.
	Let $p$ be an arbitrary optimal potential.
Suppose that there is a node $i^*$ such that $p_{i^*} \in G_s$ and $d(p_{i^*},O) > n A$. 
Then there is a subpath $P$ of $G_s$ 
that has $2 A$ edges and no node $i$ with $p_i \in P$.	
Let $U (\ni i^*)$ be the set of nodes $i$ 
such that $p_i$ is beyond $P$ (on the opposite side to $O$).
For each $i \in U$,
replace $p_i$ by its neighbor closer to $O$. 
Then one can see that $\omega$ does not increase. 
By repeating this process,
we obtain an optimal potential as required.
\end{proof}
This algorithm can be improved to a polynomial time algorithm 
by a {\em domain scaling} technique.
For a scaling parameter $\ell = -1,0,1,2,\ldots, \lceil \log n A \rceil$,
let $G_{\ell}$ denote the graph on the subset of vertices $x$ of $G$ 
with $d(x,O) \in 2^{\ell} \ZZ$, where an edge exists between $x,y \in G_\ell$ 
if and only if $d(x,y) = 2^\ell$.
By modifying 
the restriction of $\omega$ to ${G_{\ell}}^n$, 
we can define $\omega_\ell: {G_\ell}^n \to \overline{\RR}$ with the following properties:
\begin{itemize}
	\item $\omega_\ell$ is L-convex on ${G_\ell}^n$.
	\item $\omega_{-1} = \omega$.
	\item If $x^*_\ell$ is a minimizer of $\omega_\ell$ over $G_\ell \subseteq G_{\ell-1}$, then $d_{\Delta}(x^{*}_\ell, \opt(\omega_{\ell-1})) = O(n)$, where $d_{\Delta}$ is defined for ${G_{\ell-1}}^n$ (with unit edge-length).
\end{itemize}
The key is the third property, 
which comes from a {\em proximity theorem} of L-convex functions~\cite{HH14extendable}.  
By these properties, a minimizer of $\omega$ can be found
by calling SDA $\lceil \log n A \rceil$ times, 
in which $x^*_\ell$ is obtained in $O(n)$ iterations in each scaling phase.
To solve a local $k$-submodular minimization problem for $\omega_{\ell}$, 
we use a network construction in \cite{IWY14}, 
different from $D_{p}$.
Then we obtain the algorithm in Theorem~\ref{thm:algo1}.

\subsection{Node-capacitated multiflow and node-multiway cut}\label{sub:multiway}
\subsubsection{Problem formulation}
Suppose that the network $N=(V,E,b,S)$ has 
a node-capacity $b: V \setminus S \to \RR_+$ instead of edge-capacity $c$, 
where a multiflow $f:{\cal P} \to \RR_+$ should satisfy 
the node-capacity constraint:
\begin{equation}\label{eqn:node_cap}
\sum \{ f(P) \mid  P \in {\cal P}: \mbox{$P$ contains node $i$} \} \leq b(i)
\quad (i \in V \setminus S).
\end{equation}
Let $n := |V|$, $m:= |E|$, and $k := |S|$ as before.
The {\em node-capacitated maximum multiflow problem (NMF)} asks to a find a multiflow $f:{\cal P} \to \RR_+$ of 
the maximum total flow-value 
$\sum_{P \in {\cal P}} f(P)$.
This problem first appeared in 
the work of Garg, Vazirani, and Yannakakis~\cite{GVY04} 
on the node-multiway cut problem.
A {\em node-multiway cut} 
is a node subset $X \subseteq V \setminus S$ such that every $S$-path meets $X$. 
The {\em minimum node-multiway cut problem (NMC)} is 
the problem of finding a node-multiway cut $X$ of minimum capacity $\sum_{i \in X} b (i)$.
The two problems NMF and NMC are closely related.
Indeed, consider the LP-dual of NMF, which is given by
\begin{eqnarray*}
{\rm Min.}  && \sum_{i \in V \setminus S} b(i) w(i) \\
{\rm s.t.} && \sum \{ w(i) \mid \mbox{$i \in V \setminus S$: $P$ contains node $i$} \} \geq 1 \quad  (P: \mbox{$S$-path}), \\
 && \hspace{4.76cm} w(i) \geq 0 \quad  (i \in V \setminus S).
\end{eqnarray*}
Restricting $w$ to be 0-1 valued,  
we obtain an IP formulation of NMC.  
Garg, Vazirani, and Yannakakis~\cite{GVY04} proved the half-integrality of this LP, 
and showed a 2-approximation algorithm for NMC 
by rounding a half-integral solution; see also \cite{Vazirani}.
The half-integrality of the primal problem NMF was shown by Pap~\cite{Pap07STOC,Pap08EGRES}; 
this result is used to solve the integer version of NMF in strongly polynomial time.
These half-integral optimal solutions are obtained by 
the ellipsoid method in strongly polynomial time.

It is a natural challenge to develop  an ellipsoid-free algorithm.
Babenko~\cite{Babenko10} developed a combinatorial 
polynomial time algorithm for NMF in the case of unit capacity. 
For the general case of capacity, 
Babenko and Karzanov~\cite{BK08ESA} developed 
a combinatorial weakly polynomial time algorithm for NMF. 
As an application of L-convex functions 
on a twisted tree-grid,  
the paper \cite{HH15node_multi} developed the first strongly polynomial time combinatorial algorithm:
\begin{Thm}[\cite{HH15node_multi}]\label{thm:algo2}
	A half-integral optimal multiflow for NMF, 
	a half-integral optimal solution for its LP-dual, 
	and a $2$-approximate solution for NMC can be obtained in 
	$O(m (\log k) {\rm MSF}(n,m,1))$ time.
\end{Thm}
The algorithm uses a submodular flow algorithm as a subroutine.  
Let ${\rm MSF}(n,m,\gamma)$ denote the time complexity of 
solving the {\em maximum submodular flow problem} 
on a network of $n$ nodes and $m$ edges, where $\gamma$ is the time complexity of computing the {\em exchange capacity} of the defining submodular set function.
We briefly summarize the submodular flow problem; see \cite{FrankBook, FujiBook} for detail. 
A {\em submodular set function} on a set $V$
is a function $h: 2^V \to \overline\RR$ satisfying
\begin{equation*}
h(X) + h(Y) \geq h(X \cap Y) + h(X \cup Y) \quad (X,Y \subseteq V).
\end{equation*}
Let $N = (V, A, \underline{c}, \overline{c})$ be a directed network 
with lower and upper capacities $\underline{c}, \overline{c}: A \to \RR$, and
let $h:2^V \to \RR$ be a submodular set function on $V$ with $h(\emptyset) = h(V) = 0$.
A {\em feasible flow} with respect to $h$ is a function 
$\varphi: A \to \RR$ satisfying
\begin{eqnarray*}\underline{c}(e) \leq \varphi(e) 
\leq \overline{c}(e) \hspace{0.1cm} &\quad& (e \in A), \\
\varphi (\delta^-(X)) - \varphi (\delta^+(X)) \leq h(X) &\quad& (X \subseteq V). 
\end{eqnarray*}
For a feasible flow $\varphi$ and a pair of nodes $i,j$, 
the {\em exchange capacity} is defined as the minimum of
\[
h(X) - \varphi (\delta^-(X)) + \varphi (\delta^+(X)) \quad (\geq 0)
\]
over all $X \subseteq V$ with $i \in X \not \ni j$.

The maximum submodular flow problem (MSF) asks 
to find a feasible flow $\varphi$ having maximum $\varphi(e)$ for a fixed edge $e$.
This problem obviously generalizes the maximum flow problem. 
Generalizing existing maximum flow algorithms, 
several combinatorial algorithms for MSF 
have been proposed; see \cite{FI00survey} for survey.
These algorithms assume an oracle of 
computing the exchange capacity (to construct the residual network).
The current fastest algorithm for MSF is 
the pre-flow push algorithm by Fujishige-Zhang~\cite{FZ92}, 
where the time complexity is $O(n^3\gamma)$.
Thus, by using their algorithm, 
the algorithm in Theorem~\ref{thm:algo2} 
runs in $O(m n^3 \log k)$ time.

\subsubsection{Algorithm}

Let $N = (V,E,b,S)$ be a network.
For several technical reasons, instead of NMF, 
we deal with a {\em perturbed} problem.
Consider a small uniform edge-cost on $E$. 
It is clear that the objective function of NMF 
may be replaced by
$\sum_{P \in {\cal P}} M f(P) - \sum_{e \in E} 2 f(e)$ 
for large $M > 0$. 
We further purturbe $M$ 
according to terminals which $P$ connects.
Let $\Sigma$ be a tree such that each non-leaf vertex has degree 3,  
leaves are $u_s$ $(s \in S)$, and
the diameter is at most $\lceil \log k \rceil$.
For each $s \in S$, consider an infinite path $P_s$ with one end vertex $u_s'$, and
glue $\Sigma$ and $P_s$ by identifying $u_s$ and $u'_s$.
The resulting tree is denoted by $G$, where the edge-length is defined as $1$ uniformly, 
and the path-metric on $G$ is denoted by $d$.
Let $v_s$ denote the vertex in $P_s$ with $d(u_s',v_s) = (2 |E| + 1) \lceil \log k \rceil$.
The perturbed problem PNMM is to maximize
\begin{equation*}
\sum_{P \in {\cal P}} d(v_{s_P}, v_{t_P}) f(P) - \sum_{e \in E} 2 f(e)
\end{equation*}
over all multiflows $f:{\cal P} \to \RR_+$, 
where $s_P$ and $t_P$ denote the ends of an $S$-path $P$.

\begin{Lem}[{\cite{HH15node_multi}}]
Any optimal multiflow for PNMF
is optimal to NMF.
\end{Lem}

Next we explain a combinatorial duality of PNMM, 
which was earlier obtained by \cite{HH13tree_shaped}
for more general setting.
Consider the {\em edge-subdivision} $G^*$ of $G$, 
which is obtained from $G$
by replacing each edge $pq$ 
by a series of two edges $pv_{pq}$ and $v_{pq}q$ with a new vertex $v_{pq}$. 
The edge-length of $G^*$ is defined as $1/2$ uniformly, where
$G$ is naturally regarded as an isometric subspace of $G^*$ (as a metric space). 
Let $\ZZ^* := \{ z/2 \mid z \in \ZZ\}$ denote the set of half-integers.
Suppose that $V = \{1,2,\ldots,n\}$.
Consider the product $(G^* \times \ZZ^*)^n$.
An element $(p,r) = ((p_1,r_1),(p_2,r_2),\ldots,(p_n,r_n)) \in (G^* \times \ZZ^*)^n = (G^*)^n \times (\ZZ^*)^n$ is called a {\em potential} if
\begin{eqnarray*}
r_i \geq 0 \hspace{0.69cm} &\quad &  (i \in V), \\
d(p_i,p_j) - r_i - r_j \leq 2 \hspace{0.69cm} &\quad& (ij \in E), \\ 
(p_s,r_s) = (v_s,0) &\quad & (s \in S).
\end{eqnarray*}
and 
each $(p_i,r_i)$ belongs to $G \times \ZZ$ or 
$(G^* \setminus G) \times (\ZZ^* \setminus \ZZ)$.
Corresponding to Proposition~\ref{prop:duality1}, the following holds:
\begin{Prop}[{\cite{HH13tree_shaped}}]
	The optimal value of PNMF is equal to the minimum of  
$\displaystyle \sum_{i \in V \setminus S} 2 b(i) r_i$
over all potentials $(p,r)$.
\end{Prop}
A vertex $(v_{pq}, z + 1/2)$ in $(G^* \setminus G) \times (\ZZ^* \setminus \ZZ)$
corresponds to a 4-cycle $(p,z),(p,z+1),(q,z+1),(q,z)$ in $G \times \ZZ$.
Thus any potential 
is viewed as a vertex of 
a twisted tree-grid $(G \boxtimes \ZZ)^n$.
Define $\varpi:(G \boxtimes \ZZ)^n \to \overline{\RR}$ by
\begin{equation*}
\varpi(p,r) := \sum_{i \in V \setminus S} 2 b(i) r_i + I(p,r) \quad 
((p,r) \in (G \boxtimes \ZZ)^n),
\end{equation*}
where $I$ is the indicator function of the set of all potentials.
\begin{Thm}[{\cite{HH15node_multi}}]
$\varpi$ is L-convex on $(G \boxtimes \ZZ)^n$.	
\end{Thm}	

As in the previous subsection, 
we are going to apply the SDA framework to $\varpi$, 
and show that a steepest direction at $(p,r) \in (G \boxtimes \ZZ)^n$ can be obtained 
by solving a submodular flow problem 
on network $D_{p,r}$ associated with~$(p,r)$.
The argument is 
parallel to the previous subsection but is technically more complicated.

Each vertex $u \in G$ has two or three neighbors in $G$, 
which are denoted by $u^\alpha$ for $\alpha =1,2$ or $1,2,3$.
Accordingly, the neighbors in $G^*$
are denoted by $u^{*\alpha}$,   
where $u^{*\alpha}$ is the vertex replacing edge $uu^{\alpha}$.
Let $G_3 \subseteq G$ 
denote the set of vertices having three neighbors. 

Let $(p,r)$ be a potential.
Let $E_{=}$ denote the set of edges $ij$ 
with $d(p_i,p_j) - r_i - r_j = 2$.
Remove other edges.
For each nonterminal node $i$, 
replace $i$ by two nodes $i^1,i^2$ if $p_i \not \in G_3$ 
and by three nodes $i^1,i^2,i^3$ 
if  $p_i \in G_3$.
Add new edges $i^\alpha i^\beta$ for distinct $\alpha,\beta$.
The set of added edges is denoted by $E_{-}$.
For $ij \in E_{=}$, 
replace each edge $ij \in E_{=}$ 
by $i^\alpha j^\beta$ for $\alpha,\beta \in \{1,2,3\}$ with
$d(p_i, p_j) = d(p_i,(p_i)^{*\alpha}) + d((p_i)^{*\alpha}, (p_j)^{*\beta}) + d((p_j)^{*\beta}, p_j)$.
Since $d(p_i,p_j) \geq 1$ and $G^*$ is a tree, 
such neighbors $(p_i)^{*\alpha}$ and $(p_j)^{*\beta}$ are uniquely determined.
If $i = s \in S$, the incidence of $s$ is unchanged, i.e., 
let $i^\alpha = s$ in the replacement.
An edge-weight $\psi: E_{=} \cup E_{-} \to \RR$ is called 
a {\em $(p,r)$-feasible support} 
if it satisfies
\begin{eqnarray}
\psi(e) \geq 0  \hspace{0.39cm} && (e \in E_{=}), \label{eqn:psi=>0}\\
\psi(e) \leq 0 \hspace{0.39cm} && (e \in E_{-}),  \label{eqn:psi<=0}\\
 \psi (\delta(i^{\alpha})) = 0 \hspace{0.39cm} && (\alpha \in \{1,2,3\}), \label{eqn:psi(delta)=0}\\
 - \psi(i^1 i^2) \leq b(i) && 
 (p_i \not \in G_3, r_i = 0), \label{eqn:node_cap2=}\\
   -\psi(i^1i^2) = b(i) && (p_i \not \in G_3, r_i > 0), \label{eqn:node_cap2>}\\  
 - \psi(i^1i^2) - \psi(i^2i^3) - \psi(i^1i^3) \leq b(i) &&
(p_i \in G_3, r_i = 0),\label{eqn:node_cap3=}\\
 - \psi(i^1i^2) - \psi(i^2i^3) - \psi(i^1i^3) = b(i) && 
 (p_i \in G_3, r_i > 0) \label{eqn:node_cap3>}
\end{eqnarray}
for each edge $e$ and nonterminal node $i$.
By precisely the same argument,  
a $(p,r)$-feasible support $\psi$ is decomposed 
as $\psi = \sum_{P \in {\cal P}} f^{\psi}(P) \chi_P$ 
for a multiflow $f^{\psi}:{\cal P} \to \RR_+$,
where the node-capacity constraint (\ref{eqn:node_cap}) follows from (\ref{eqn:node_cap2=})-(\ref{eqn:node_cap3>}).
Corresponding to Lemma~\ref{lem:optimality1}, 
we obtain the following, where the proof goes along the same argument. 
\begin{Lem}[\cite{HH15node_multi}]
	\begin{itemize}
	\item[{\rm (1)}]\ A potential $(p,r)$ is optimal 
	if and only if a $(p,r)$-feasible support exists.
	\item[{\rm (2)}]\ For any $(p,r)$-feasible support $\psi$, 
	the multiflow $f^{\psi}$ is optimal to PNMF.
	\end{itemize}
\end{Lem}
The system of inequalities (\ref{eqn:psi=>0})-(\ref{eqn:node_cap3>}) is similar to the previous bidirected flow problem (\ref{eqn:p-feasible_cap1})-(\ref{eqn:p-feasible_conser}). 
However (\ref{eqn:node_cap3=}) and (\ref{eqn:node_cap3>}) 
are not bidirected flow constraints.
In fact, 
the second constraint (\ref{eqn:node_cap3>}) reduces to 
a bidirected flow constraint as follows.
Add new vertex $i^0$, replace edges $i^1i^2, i^2i^3,i^1i^3$ 
by $i^0i^1, i^0i^2,i^0i^3$, and replace (\ref{eqn:node_cap3>}) by
\begin{equation}\label{eqn:node_cap3>'}
- \psi(\delta (i^0)) = 2 b(i),\quad  - \psi(i^0i^\alpha) \leq b(i) \ (\alpha =1,2,3). 
\end{equation} 
Then 
$(\psi(i^1 i^2), \psi(i^2 i^3),\psi(i^1 i^3))$ 
 satisfying (\ref{eqn:node_cap3>}) is represented
 as 
 \[
 \psi(i^{\alpha}i^{\beta}) = (\psi(i^0i^\alpha) + \psi(i^{0}i^\beta) - \psi(i^{0}i^\gamma))/2\quad 
 (\{ \alpha,\beta,\gamma\} = \{1,2,3\})
 \]
 for 
 $(\psi(i^0i^1),\psi(i^0i^2),\psi(i^0i^3))$ satisfying (\ref{eqn:node_cap3>'}).

We do not know whether (\ref{eqn:node_cap3=}) 
admits such a reduction. 
This makes the problem difficult.
Node $i \in V \setminus S$ with $p_i \in G_3, r_i= 0$ 
is said to be {\em special}.
For a special node $i$, 
remove edges $i^1i^2,i^2i^3,i^1i^3$.
Then the conditions  (\ref{eqn:psi(delta)=0})  and (\ref{eqn:node_cap3=}) 
for $u = i^1,i^2,i^3$ can be equivalently written 
as the following condition on degree vector 
$(\psi(\delta(i^1)),\psi(\delta(i^2)), \psi(\delta(i^3)))$:
\begin{equation}\label{eqn:bisub'}
\psi(\delta(X)) - \psi(\delta(Y)) \leq g_i (X,Y) \quad 
(X,Y \subseteq \{i^1,i^2,i^3\}: X \cap Y = \emptyset),
\end{equation}
where $g_i$ is a function 
the set of pairs $X,Y \subseteq \{i^1,i^2,i^3\}$ 
with $X \cap Y \neq \emptyset$.
Such $g_i$ can be chosen as a bisubmodular (set) function, 
and inequality system (\ref{eqn:bisub'}) 
says that the degree vector must 
belong to the {\em bisubmodular polyhedron} associated with $g_i$; 
see \cite{FujiBook} for bisubmodular polyhedra.
Thus our problem of solving  (\ref{eqn:psi=>0})-(\ref{eqn:node_cap2>}), (\ref{eqn:node_cap3>'}), and (\ref{eqn:bisub'}) is a fractional bidirected flow problem 
with degrees constrained by a bisubmodular function, 
which may be called a {\em bisubmodular flow} problem. 
However this natural class of the problems has not been well-studied so far.
We need a further reduction. 
As in the previous subsection, 
for the bidirected network associated with (\ref{eqn:psi=>0})-(\ref{eqn:node_cap2>}), and (\ref{eqn:node_cap3>'}),  
we construct the equivalent directed network $D_{p,r}$ 
with upper capacity $\overline{c}$ and lower capacity $\underline{c}$, where node $i^{\alpha}$ and edge $e$ are doubled as 
$i^{\alpha +}$, $i^{\alpha-}$ and  $e^{+}$, $e^{-}$, respectively. 
Let $V_i$ denote the set of nodes replacing original $i \in V$, 
as before.
We construct a submodular-flow constraint on 
$V_i = \{i^{1+},i^{2+},i^{3+},i^{1-},i^{2-},i^{3-}\}$ 
for each special node $i$  
so that a $(p,r)$-feasible support $\psi$ 
is recovered from any feasible flow $\varphi$ on $D_{p,r}$ 
by the following relation:
\begin{equation}\label{eqn:recover}
\psi(e) = (\varphi(e^+) + \varphi(e^-))/2.
\end{equation}
Such a submodular-flow constraint actually exists, and is represented by some submodular set function $h_i$ on $V_i$; see \cite{HH15node_multi} for the detailed construction of $h_i$.

Now our problem is to find a flow $\varphi$ on network $D_{p,r}$ 
having the following properties:
\begin{itemize}
		\item  For each edge $e$ in $D_{p,r}$, it holds that $\underline{c}(e) \leq \varphi(e) \leq \overline{c}(e)$. 
		\item  For each special node $i$, it holds that
			\begin{equation}\label{eqn:submo'}
			\varphi(\delta^-(U)) - \varphi(\delta^+(U) )\leq h_i (U) \quad (U \subseteq V_i).
			\end{equation}
		\item For other node $i \in V \setminus S$, it holds that
		\begin{equation*}
		\varphi (\delta^- (i^{\alpha \sigma})) - \varphi (\delta^+(i^{\alpha \sigma})) = 0 \quad (\alpha \in \{0,1,2,3\},\sigma \in \{+,-\}).
		\end{equation*}
\end{itemize} 
This is a submodular flow problem, 
where the defining submodular set function is given by 
$X \mapsto \sum_{i: {\rm special}} h_i (X \cap V_i)$, 
and its exchange capacity is computed in constant time.
Consequently we have the following, where 
the half-integrality follows from 
the integrality theorem of submodular flow.
\begin{Lem}[\cite{HH15node_multi}]
	From an optimal potential $(p,r)$, 
	a half-integral optimal multiflow is obtained in $O({\rm MSF}(n,m,1))$ time.
\end{Lem}
By Frank's theorem on the feasibility of submodular flow (see \cite{FrankBook,FujiBook}),  
a feasible flow $\varphi$ exists if and only if
\begin{equation}
\kappa(X) := \underline{c}(\delta^-(X)) - \overline{c}(\delta^+(X)) - 
\sum_{i: {\rm special}} h_i (X \cap V_i).
\end{equation}
is nonpositive for every vertex subset $X$.
A {\em violating cut} is a vertex subset $X$ having positive $\kappa(X)$.
By a standard reduction technique, 
finding a feasible flow or violating cut is reduced 
to a maximum submodular flow problem, 
where a minimal violating cut is naturally obtained from 
the residual graph of a maximum feasible flow. 

Again,
we can obtain a steepest direction from a minimal violating cut,  
according to its intersection pattern with each $V_i$.
A vertex subset $X$ is called {\em movable} 
if for $i$ with $p_i \in G^* \setminus G$ 
it holds that $|X \cap V_i| \leq 1$ and 
for node $i$ with $p_i \in G$ it holds that
$X \cap V_i = \emptyset$, $V_i^+$, $V_i^-$, $\{i^{\alpha+}\}$, 
$V_i^{-} \setminus \{i^{\alpha-}\}$, or 
$\{i^{\alpha+}\} \cup V_i^{-} \setminus \{i^{\alpha-}\}$
for some $\alpha \in \{1,2,3\}$.
For a movable cut $X$, define $(p,r)^X$ by
\begin{equation}
(p,r)^X_i 
:= \left\{
\begin{array}{ll}
(p_i,r_i) & {\rm if}\ X \cap V_i = \emptyset,\\
(p_i^{*\alpha},r_i+1/2) & {\rm if}\ X \cap V_i = \{i^{\alpha+}\},\\
(p_i^{*\alpha},r_i-1/2) & {\rm if}\ X \cap V_i =V_i^{-} \setminus \{i^{\alpha-}\},\\
(p_i^{\alpha},r_i) & {\rm if}\ X \cap V_i = \{i^{\alpha+}\} \cup V_i^{-} \setminus \{i^{\alpha-}\},\\
(p_i,r_i+1) & {\rm if}\ X \cap V_i = V_i^+,\\
(p_i,r_i-1) & {\rm if}\ X \cap V_i = V_i^-.\\
\end{array}
\right.
\end{equation}
See Figure~\ref{fig:cut2} for an intuition of $(p,r)^X$.
\begin{figure}[t]
	\begin{center}
		\includegraphics[scale=1.0]{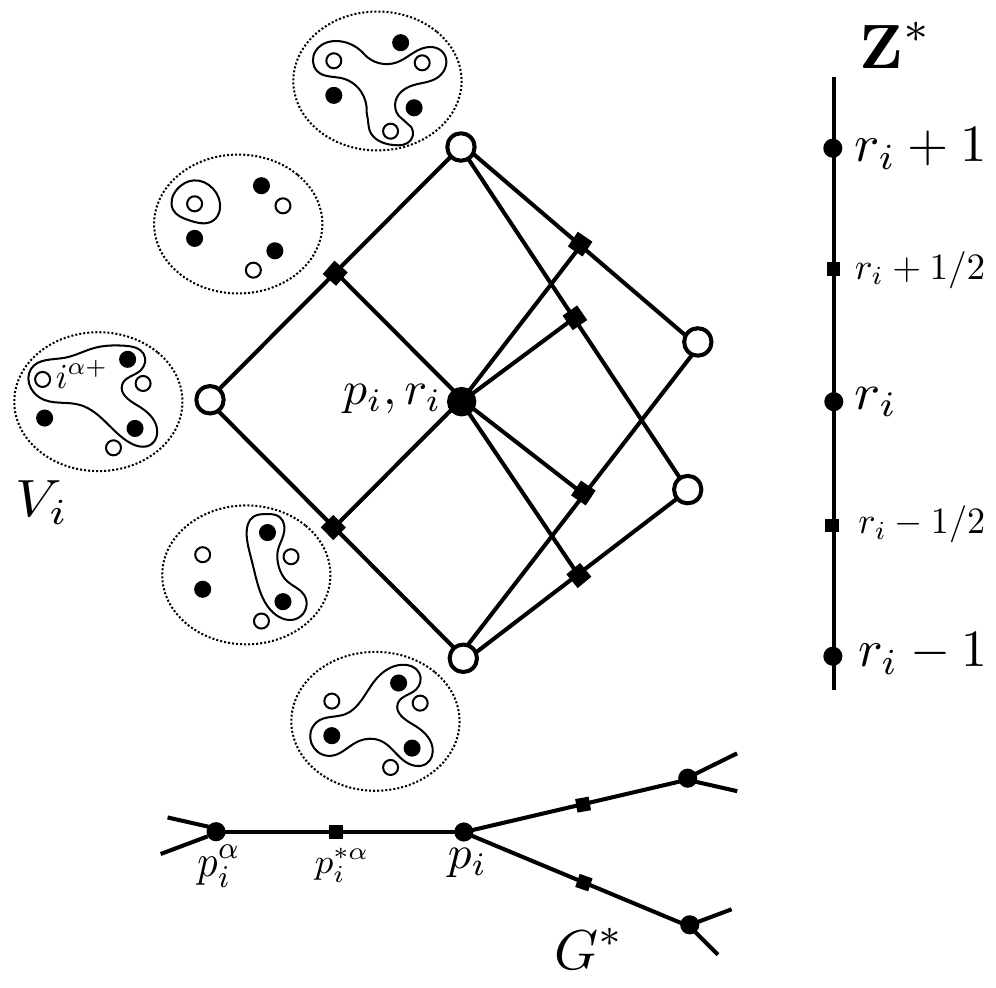}
		\caption{The correspondence between movable cuts in $D_{p,r}$ and neighbors of $(p,r)$. 
			For $(p_i,r_i) \in B$,
			there is a one-to-one correspondence between 
			$F_{p_i}$ and $\{ X \cap V_i \mid X: \mbox{movable cut}\}$,  
		where the meaning of this figure is the same as in Figure~\ref{fig:cut1}. 
		}
		\label{fig:cut2}
	\end{center}
\end{figure}\noindent
Let $V_F$ be the union of $V_i$ over 
$i \in V$ with $(p_i,r_i) \in B$ and 
$\{ i^{\alpha+}, i^{\alpha-} \}$ 
over $i \in V$ and $\alpha \in \{1,2\}$ with 
$p_i \in G^* \setminus G$ and $(p_i^{*\alpha}, r_i-1/2) \in B$.
Let $V_I$ be defined by replacing the role of $B$ and $W$ in $V_F$. Edges between $V_I$ and $V_F$ have the same (finite) lower and upper capacity. 
If $X$ is violating, then $X \cap V_I$ or $X \cap V_F$ is violating.
\begin{Lem}[\cite{HH15node_multi}]
Let $X$ be the unique minimal maximum violating cut, and let $\tilde X$ be obtained from $X$ by
adding $V_i^+$ for each node $i \in V \setminus S$ 
with $p_i \in G_3$ and $|X \cap V_i^+|=2$.
Then $\tilde X \cap V_I$ and $\tilde X \cap V_F$ are movable, one of them is violating, and
$(p,r)^{\tilde X \cap V_I}$ is a minimizer of $\varpi$ over $I_{p,r}$ 
and $(p,r)^{\tilde X \cap V_F}$ is a minimizer of $\varpi$ over $F_{p,r}$. 
\end{Lem}
Now the algorithm to find an optimal multiflow is given as follows.
\begin{description}
	\item[{\bf Steepest Descent Algorithm for PNMF}]
	\item[Step 0:] For each terminal $s \in S$, 
	let $(p_s,r_s) := (v_s,0)$. 
	Choose any vertex $v$ in $\Sigma$. 
	For each $i \in V \setminus S$, 
	let $p_i := v$ and $r_i := 2(m+1) \lceil \log k \rceil$. 
	\item[Step 1:] Construct $D_{p,r}$ 
	with submodular set function $X \mapsto \sum_{i:{\rm special}} h_i (X \cap V_i)$.
	\item[Step 2:] If a feasible flow $\varphi$ exists in $D_{p,r}$, 
	then obtain a $(p,r)$-feasible support $\psi$ from $\varphi$
	and an optimal multiflow $f^{\psi}$ via the path decomposition; stop.
	\item[Step 3:] For a minimal maximum violating cut $X$,   
	choose $(p',r') \in \{ (p,r)^{\tilde X \cap V_I}, (p,r)^{\tilde X \cap V_F} \}$ with 
	$\varpi(p',r') = \min \{ \varpi((p,r)^{\tilde X \cap V_I}), \varpi((p,r)^{\tilde X \cap V_F})\}$, 
	let $(p,r) := (p',r')$, and go to step~1.
\end{description}
One can see that the initial point is actually a potential.
By the argument similar to the proof of Proposition~\ref{prop:runin}, one can show that 
there is an optimal potential $(p^*,r^*)$ such that $r^*_i= O(m \log k)$ and 
$d(v,p_i^*) = O(m \log k)$.
Consequently, the number of iterations is bounded by $O(m \log k)$, 
and we obtain Theorem~\ref{thm:algo2}.

\begin{Rem}{\rm 
		As seen above, 
	${\pmb k}$- and $({\pmb k}, {\pmb l})$-submodular functions arising 
	from localizations of $\omega$ and $\varpi$ 
	can be minimized via maximum (submodular) flow.
	A common feature of both cases is that 
	the domain $S_{\pmb k}$ or $S_{{\pmb k}, {\pmb l}}$ 
	of a ${\pmb k}$- or $({\pmb k}, {\pmb l})$-submodular function
	is associated with 
	special intersection patterns between nodes and cuts on which 
	the function-value is equal to the cut-capacity (up to constant).
	A general framework 
	for such network representations  
	is discussed by Iwamasa~\cite{Iwamasa}.  
	}
\end{Rem}

\section{L-convex function on oriented modular graph}\label{sec:oriented_modular}

In this section, we explain L-convex functions 
on oriented modular graphs, introduced in~\cite{HH150ext,HH16L-convex}.
This class of discrete convex functions is 
a further generalization of L-convex functions 
in Section~\ref{sec:gridstructures}.
The original motivation of our theory comes from the complexity  classification of the minimum 0-extension problem.
We start by mentioning the motivation and highlight 
of our theory (Section~\ref{subsec:motivation}), 
and then go into the details (Sections~\ref{subsec:submo_modular} and \ref{subsec:L-convex}).

\subsection{Motivation: Minimum 0-extension problem}\label{subsec:motivation} 
Let us introduce
the {\em minimum 0-extension problem (0-EXT)}, 
where our formulation is different from but equivalent to the original formulation 
by Karzanov~\cite{Kar98a}.
An input $I$ consists of 
number $n$ of variables, undirected graph $G$, 
nonnegative weights $b_{iv}$ $(1 \leq i \leq n, v \in G)$ and $c_{ij}$ $(1 \leq i < j \leq n)$. 
The goal of 0-EXT is to find $x = (x_1,x_2,\ldots, x_n) \in G^n$ 
that minimizes 
\begin{equation}
\sum_{i=1}^n \sum_{v \in G} b_{iv} d(x_i, v) + \sum_{1 \leq i < j \leq n} c_{ij} d(x_i, x_j), 
\end{equation}
where $d = d_G$ is the shortest path metric on $G$.
This problem is interpreted as a facility location on graph $G$.
Namely
we are going to locate 
new facilities $1,2,\ldots,n$ on graph $G$ of cities,
where these facilities communicate each other and 
communicate with all cities, and communication costs 
are propositional to their distances.
The problem is to find 
a location of minimum communication cost.
In facility location theory~\cite{TFL83}, 0-EXT is known as 
the {\em multifacility location problem}. 
Also 0-EXT is an important special case of 
the {\em metric labeling problem}~\cite{KleinbergTardos02}, 
which is
a unified label assignment problem 
arising from computer vision and machine learning.
Notice that
fundamental combinatorial optimization problems can be formulated as 0-EXT for special underlying graphs.
The minimum cut problem is the case of $G = K_2$, and the multiway cut problem is the case of $G = K_m$ $(k \geq 3)$.

In \cite{Kar98a}, Karzanov addressed 
the computational complexity of 
0-EXT with fixed underlying graph $G$.
This restricted problem class is denoted by 0-EXT$[G]$.
He raised a question: 
{\em What are graphs $G$ for which {\rm 0-EXT}$[G]$ is polynomially solvable?} 
An easy observation is that
0-EXT$[K_m]$ is in P if $m \leq 2$ and NP-hard otherwise.
A classical result~\cite{Kolen} in facility location theory is  
that 0-EXT$[G]$ is in P for a tree $G$.
Consequently, 0-EXT$[G]$ is in P for a tree-product $G$.
It turned out that the tractability of 0-EXT is strongly linked 
to median and modularity concept of graphs.
A {\em median} of three vertices $x_1,x_2,x_3$ 
is a vertex $y$ satisfying 
\begin{equation*}
d(x_{i},x_{j}) = d(x_i,y) + d(y,x_j) \quad (1 \leq i < j \leq 3).
\end{equation*}
A median is a common point 
in shortest paths among the three points, 
may or may not exist, and is not necessarily unique even if it exists.
A {\em median graph} is a connected graph 
such that every triple of vertices has a {\em unique} median.
Observe that trees and their products are median graphs. 
Chepoi~\cite{Chepoi96} and Karzanov~\cite{Kar98a} independently showed that 0-EXT$[G]$ is in P for a median graph $G$.

A {\em modular graph} is a further generalization of a median graph, 
and is defined as a connected graph such that 
every triple of vertices admits (not necessarily unique) a median.
The following hardness result shows that  
graphs tractable for 0-EXT are necessarily modular. 
\begin{Thm}[{\cite{Kar98a}}]\label{thm:hardness}
	If $G$ is not orientable modular, 
	then {\rm 0-EXT}$[G]$ is NP-hard.
\end{Thm}
Here a (modular) graph is said to be {\em orientable}
if it has an edge-orientation, called an {\em admissible orientation}, such that
every 4-cycle $(x_1,x_2,x_3,x_4)$ is oriented
as: $x_1 \to x_2$ if and only if $x_4 \to x_3$.
Karzanov~\cite{Kar98a,Kar04a} showed that 0-EXT$[G]$ is polynomially solvable on
special classes of orientable modular graphs.
In \cite{HH150ext}, we proved the tractability 
for general orientable modular graphs. 
\begin{Thm}[\cite{HH150ext}]\label{thm:0-EXT}
	If $G$ is orientable modular,
	then {\rm 0-EXT}$[G]$ is solvable in polynomial time.
\end{Thm}
For proving this result, \cite{HH150ext} introduced L-convex functions on oriented modular graphs and 
submodular functions on modular semilattices, 
and applied the SDA framework to 0-EXT.
An {\em oriented modular graph} is 
an orientable modular graph endowed 
with an admissible orientation.
A {\em modular semilattice} is a semilattice generalization 
of a modular lattice, introduced by Bandelt, Van De Vel, and Verheul~\cite{BVV93}. 
Recall that a modular lattice $L$ is a lattice such that 
for every $x,y,z \in L$ with $x \succeq z$ it holds $x \wedge (y \vee z) = (x \wedge y) \vee z$.
A modular semilattice
is a meet-semilattice $L$ 
such that every principal ideal is a modular lattice, 
and for every $x,y,z \in L$ 
the join $x \vee y \vee z$ exists provided 
$x \vee y$, $y \vee z$, and $z \vee x$ exist.
These two structures 
generalize Euclidean buildings of type C and polar spaces, respectively, and are related in the following way.
\begin{Prop}\label{prop:modular}
	\begin{itemize}
		\item[{\rm (1)}]\ A semilattice is modular if and only if its Hasse diagram is oriented modular~{\rm \cite{BVV93}}.
		\item[{\rm (2)}]\ 
		Every principal ideal and filter of an oriented modular graph 
		are modular semilattices~{\rm \cite{HH150ext}}. In particular, every interval is a modular lattice.
		\item[{\rm (3)}]\ A polar space is a modular semilattice~{\rm \cite{CCHO14}}.
		\item[{\rm (4)}]\ The Hasse diagram of a Euclidean building of type C is oriented modular~{\rm \cite{CCHO14}}.
	\end{itemize}
\end{Prop}
An admissible orientation is acyclic~\cite{HH150ext}, and 
an oriented modular graph is viewed as (the Hasse diagram of) a poset.

As is expected from these properties 
and arguments in Section~\ref{sec:gridstructures},
an L-convex function on an oriented modular graph 
is defined so that 
it behaves submodular on the local structure 
(principal ideal and filter) 
of each vertex, which is a modular semilattice.
Accordingly, the steepest descent algorithm is well-defined, 
and correctly obtain a minimizer.

We start with the local theory in 
the next subsection (Section~\ref{subsec:submo_modular}), 
where we introduce submodular functions on modular semilattices.
Then, in Section~\ref{subsec:L-convex}, we introduce 
L-convex functions on oriented modular graphs,  
and outline the proof of Theorem~\ref{thm:0-EXT}.

\begin{Rem}\label{rem:dichotomy}{\rm 
		The minimum 0-extension problem 0-EXT$[\Gamma]$ on a fixed $\Gamma$ 
		is a particular instance of 
		{\em finite-valued CSP} 
		with a fixed language. 
		Thapper and \v{Z}ivn\'y~\cite{TZ16ACM} 
		established a dichotomy theorem for finite-valued CSPs.
		The complexity dichotomy in Theorems~\ref{thm:hardness} and \ref{thm:0-EXT} 
		is a special case of their dichotomy theorem, though a characterization of 
		the tractable class of graphs (i.e., orientable modular graphs) 
		seems not to follow directly from their result. 
	}
\end{Rem}

\subsection{Submodular function on modular semilattice}\label{subsec:submo_modular}
A modular semilattice, though not necessarily a lattice, 
admits an analogue of the join, called the {\em fractional join}, 
which is motivated by {\em fractional polymorphisms} in VCSP~\cite{KTZ13,ZivnyBook} 
and enables us to introduce a submodularity concept.

Let $L$ be a modular semilattice, 
and let $r:L \to \ZZ_+$ be the rank function, 
i.e., $r(p)$ is the length of a maximal chain from the minimum element to $p$.
The fractional join of elements $p,q \in L$ is defined as
a formal sum 
\[
\sum_{u \in E(p,q)} [C(u;p,q)] u
\]
of elements $u \in E(p,q) \subseteq L$ with nonnegative coefficients 
$[C(u;p,q)]$, to be defined soon. %
Then 
a function $f:L \to \overline\RR$ is called {\em submodular} if it satisfies 
\begin{equation*}\label{eqn:submodular}
f(p) + f(q) \geq f(p \wedge q) + \sum_{u \in E(p,q)} [C(u;p,q)] f(u)
\quad (p,q \in L).
\end{equation*}

The fractional join of $p,q \in L$ 
is defined according to the following steps; 
see Figure~\ref{fig:frac_join} for intuition.
		\begin{figure}[t]
			\begin{center}
				\includegraphics[scale=0.85]{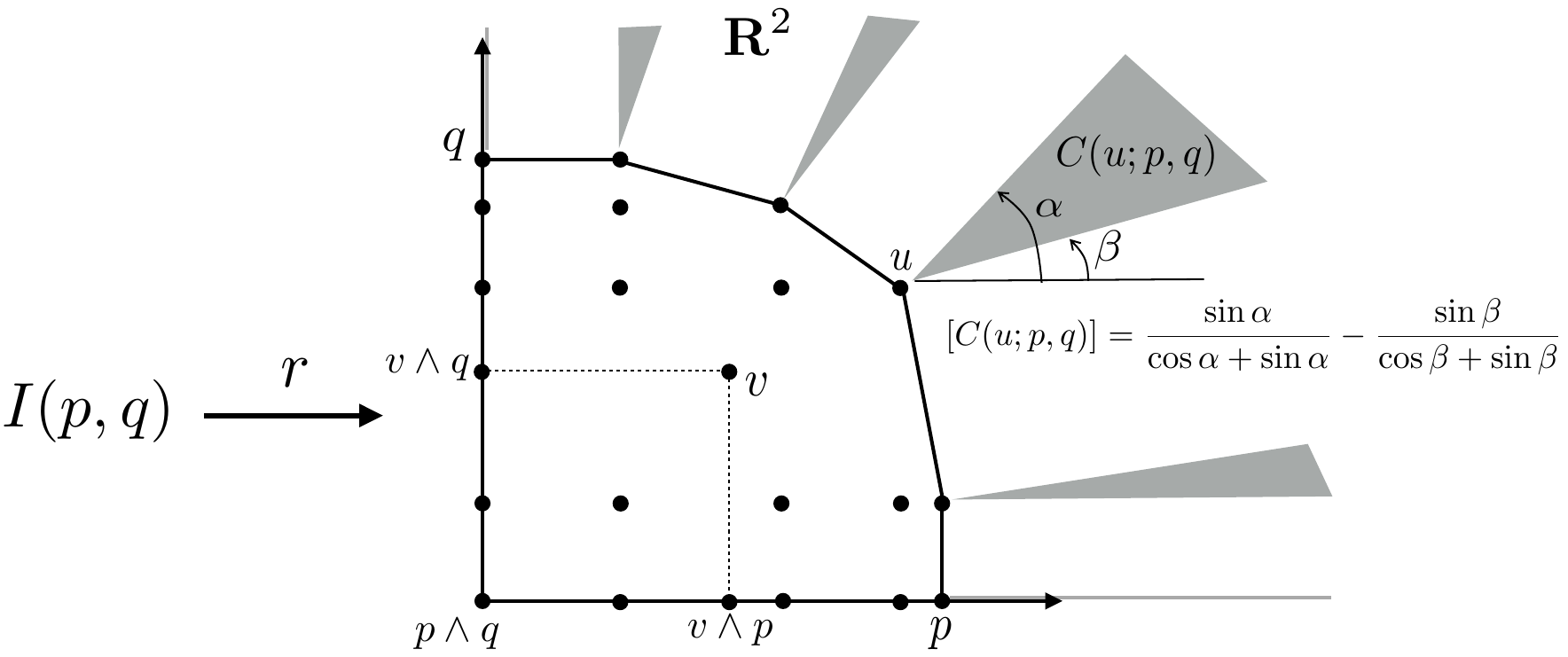}
				\caption{The construction of the fractional join. 
					By $u \mapsto r(u;p,q)$, 
					the set $I(p,q)$ is mapped to points in $\RR^2_+$, 
					where $r(p \wedge q;p,q)$ is the origin, 
					$r(p;p,q)$ and $r(q;p,q)$ are on the coordinate axes.
					Then $\Conv I(p,q)$ is the convex hull of $r(u;p,q)$ over $u \in I(p,q)$. 
					The fractional join is defined as 
					the formal sum of elements mapped to maximal extreme points of $\Conv I(p,q)$.
					}
				\label{fig:frac_join}
			\end{center}
		\end{figure}\noindent
\begin{itemize}
	\item Let $I(p,q)$ denote 
	the set of all elements $u \in L$ represented as $u = a \vee b$ 
	for some $(a,b)$ with $p \succeq a \succeq p \wedge q \preceq b \preceq q$. This representation is unique, and $(a,b)$ equals $(u \wedge p, u \wedge q)$~\cite{HH150ext}.
	
	\item For $u \in I(p,q)$, let $r(u;p,q)$ be the vector in $\RR^2_+$ defined by
	\begin{equation*}\label{eqn:r}
	r(u; p,q) = (r(u \wedge p)- r(p \wedge q), r(u \wedge q) - r(p \wedge q)).
	\end{equation*}
	\item Let $\Conv I(p,q) \subseteq \RR^2_+$ denote 
	the convex hull of vectors $r(u;p,q)$ over all $u \in I(p,q)$.
	\item Let $E(p,q)$ be the set of elements $u$ in $I(p,q)$
	such that $r(u;p,q)$ is a maximal extreme point of $\Conv I(p,q)$.
	Then $u \mapsto r(u;p,q)$ is injective on $E(p,q)$~\cite{HH150ext}.
	\item For $u \in E(p,q)$, 
	let $C(u;p,q)$ denote 
	the nonnegative normal cone at $r(u; p,q)$:
	\[
	C(u;p,q) := \{ c \in \RR_+^2 \mid \langle c, r(u;p,q) \rangle = \max_{x \in \Conv I(p,q)} \langle c, x \rangle \},
	\]
	where $\langle \cdot, \cdot \rangle$ is the standard inner product.
	\item For a convex cone $C \subseteq \RR_+^2$ represented as 
	\[
	C = \{(x,y) \in \RR_+^2 \mid y \cos \alpha \leq x \sin \alpha, y \cos \beta \geq x \sin \beta \}
	\]
	for $0 \leq \alpha \leq \beta \leq \pi/2$,  define nonnegative value $[ C ]$ by
		\[
		[C] := \frac{\sin \alpha}{\sin \alpha + \cos \alpha} - 
		\frac{\sin \beta}{\sin \beta + \cos \beta}. 
		\]
	\item The fractional join of $p,q$ is defined as 
	$\displaystyle
	\sum_{u \in E(p,q)} [C(u;p,q)] u.
    $ 
\end{itemize}
%

This weird definition of the submodularity 
turns out to be appropriate.
If $L$ is a modular lattice, then 
the fractional join is equal to the join $1\cdot \vee = \vee$, 
and our definition of submodularity coincides 
with the usual one.
In the case where $L$ is a polar space, it is shown in \cite{HH16L-convex} that the fractional join of $p,q$ is 
equal to
\[
\frac{1}{2} (p \sqcup q) \sqcup q + \frac{1}{2} (p \sqcup q) \sqcup p,
\]
and hence a submodular function on $L$ is a function satisfying 
\begin{equation}\label{eqn:submo_polar2}
f(p) + f(q) \geq f(p \wedge q) + \frac{1}{2} f((p \sqcup q) \sqcup q) 
+ \frac{1}{2} f((p \sqcup q) \sqcup p) \quad (p,q \in L).
\end{equation}
It is not difficult to see that systems 
of inequalities (\ref{eqn:submo_polar2}) 
and (\ref{eqn:submo_polar1}) define the same class of functions. 
Thus the submodularity concept
in this section is consistent with that in Section~\ref{subsec:building}.

An important property relevant to 0-EXT is its relation 
to the distance on $L$.
Let $d: L \times L \to \RR$ denote the shortest path metric on 
the Hasse diagram of $L$.
Then $d$ is also written as
\[
d(p,q) = r(p) + r(q) - 2r(p \wedge q) \quad (p,q \in L).
\]
\begin{Thm}[\cite{HH150ext}]\label{thm:distance}
	Let $L$ be a modular semilattice. 
	Then the distance function $d$ is submodular on $L \times L$.
\end{Thm}

Next we consider the minimization of 
submodular functions on a modular semilattice.
The tractability under general setting (i.e., oracle model) is unknown.
We consider a restricted situation 
of {\em valued constraint satisfaction problem (VCSP)}; 
see \cite{KTZ13,ZivnyBook} for VCSP. 
Roughly speaking, VCSP is the minimization problem 
of a sum of functions with small number of variables.
We here consider the following VCSP 
({\em submodular-VCSP on modular semilattice}).
An input consists of 
(finite) modular semilattices $L_1, L_2, \ldots, L_n$ and submodular functions 
$f_i : L_{i_1} \times L_{i_2} \times \cdots \times L_{i_{k}} \to \overline{\RR}$ with $i=1,2,\ldots,m$ and $1 \leq i_1 < i_2 < \cdots < i_{k} \leq n$, where $k$ is a fixed constant.
The goal is to find $p = (p_1,p_2,\ldots, p_n) \in L_1 \times L_2 \times \cdots \times L_n$ to minimize
\[
\sum_{i=1}^m f_i (p_{i_1}, p_{i_2},\ldots,p_{i_{k}}).
\]
Each submodular function $f_i$ is given as the table of all function values. 
Hence the size of the input is $O(n N + m N^k)$ for $N := \max_{i} |L_i|$.

Kolmogorov, Thapper, and \v{Z}iv\'{n}y~\cite{KTZ13} proved 
a powerful tractability criterion for general VCSP 
such that an LP-relaxation ({\em Basic LP}) exactly solves the VCSP instance. 
Their criterion involves the existence of 
a certain submodular-type inequality ({\em fractional polymorphism})
for the objective functions, 
and is applicable to our submodular VCSP 
(thanks to the above weird definition).
\begin{Thm}[{\cite{HH150ext}}]\label{thm:vscp}
	Submodular-VCSP on modular semilattice 
	is solvable in polynomial time.
\end{Thm}

\begin{Rem}{\rm 
Kuivinen \cite{Kuivinen09,Kuivinen11} proved 
a good characterization for general SFM 
on product $L_1 \times L_2 \times \cdots \times L_n$ of modular lattices $L_i$ with $|L_i|$ fixed.
Fujishige, Kir\'aly, Makino, Takazawa, and Tanigawa~\cite{FKMTT14} 
proved the oracle-tractability for the case 
where each $L_i$ is a diamond, i.e.,  a modular lattice of rank 2.

}
\end{Rem}

\subsection{L-convex function on oriented modular graph}\label{subsec:L-convex}

Here we introduce L-convex functions for 
a slightly restricted subclass of oriented modular graphs;
see Remark~\ref{rem:general} for general case.
Recall Proposition~\ref{prop:modular}~(2) 
that every interval of an oriented modular graph $\Gamma$ is a modular lattice.
If every interval of $\Gamma$ 
is a {\em complemented} modular lattice, i.e.,
every element is a join of rank-$1$ elements, 
then $\Gamma$ is said to be {\em well-oriented}.
Suppose that $\Gamma$ is a well-oriented modular graph.
The L-convexity on $\Gamma$ is defined 
along the property (3) in Theorem~\ref{thm:convex}, 
not by discrete midpoint convexity, 
since we do not know how to define discrete midpoint operations 
on $\Gamma$. 
Namely an L-convex function on $\Gamma$ is a function $g: \Gamma \to \RR$
such that $g$ is submodular on every principal ideal and filter, and $\dom g$ is chain-connected, 
where the chain-connectivity is similarly defined 
as in Theorem~\ref{thm:convex}.
By this definition, the desirable properties hold:
\begin{Thm}[\cite{HH150ext,HH16L-convex}]
	Let $g$ be an L-convex function on $\Gamma$.
	 If $x \in \dom g$ is not a minimizer of $g$, 
	 then there is $y \in I_x \cup F_x$ with $g(y) < g(x)$. 
\end{Thm}
Thus the steepest descent algorithm (SDA) is well-defined, 
and correctly obtains a minimizer of $g$ (if it exists).
Moreover the $l_{\infty}$-iteration bound is also generalized.
Let $\Gamma^{\Delta}$ denote the graph obtained from $\Gamma$ by adding an edge $pq$ if both $p \wedge q$ and $p \vee q$ exist in $\Gamma$, and let $d_{\Delta} := d_{\Gamma^\Delta}$.
Then Theorem~\ref{thm:bound} is generalized as follows.
\begin{Thm}[\cite{HH16L-convex}]\label{thm:bound}
	The number of iterations of SDA 
	applied to L-convex function $g$ and initial point $x \in \dom g$ 
	is at most $d_{\Delta}(x, \opt(g))+2$. 
\end{Thm}
Corresponding to Theorem~\ref{thm:distance}, 
the following holds: 
\begin{Thm}[\cite{HH150ext}]\label{thm:d_is_L-convex}
	Let $G$ be an oriented modular graph.
	The distance function $d$ on $G$ 
	is L-convex on $G \times G$.
\end{Thm}
We are ready to prove Theorem~\ref{thm:0-EXT}.
Let $G$ be an orientable modular graph.
Endow $G$ with an arbitrary admissible orientation.
Then the product $G^n$ of $G$ is oriented modular.
It was shown in \cite{HH150ext,HH16L-convex} that the class of L-convex functions is 
closed under suitable operations 
such as variable fixing, nonnegative sum, and direct sum.
By this fact and Theorem~\ref{thm:d_is_L-convex},
the objective function of 0-EXT$[\Gamma]$
is viewed as an L-convex function on 
$G^n$.
Thus we can apply the SDA framework to 0-EXT$[\Gamma]$, 
where each local problem is submodular-VCSP on modular semilattice.
By Theorem~\ref{thm:vscp}, 
a steepest direction at $x$ 
can be found in polynomial time.
By Theorem~\ref{thm:bound},
the number of iterations is 
bounded by the diameter of $(G^n)^{\Delta}$.
Notice that 
for $x,x' \in G^n$,
if $\max_i d(x_i,x_i') \leq 1$, then $x$ and $x'$ are adjacent in $(G^n)^{\Delta}$.
From this, we see that
the diameter of $(G^n)^{\Delta}$ 
is not greater than the diameter of $G$. 
Thus the minimum 0-extension problem on $G$ is solved in polynomial time.

\begin{Rem}\label{rem:general}
	{\rm 
		Let us sketch the definition of L-convex function
		on general oriented modular graph $\Gamma$.
		Consider the poset of all intervals $[p,q]$
		such that $[p,q]$ is a complemented modular lattice, 
		where the partial order is the inclusion order.
		Then the Hasse diagram $\Gamma^*$ is well-oriented modular~\cite{CCHO14,HH150ext}.
		For a function $g: \Gamma \to \overline{\RR}$, 
		let $g^*: \Gamma^* \to \overline{\RR}$ be defined 
		by $g^*([p,q]) = (g(p) + q(q))/2$.
		Then an L-convex function on $\Gamma$
		is defined as a function 
		$g: \Gamma \to \overline{\RR}$ such that 
		$g^*$ is L-convex on $\Gamma^*$.
		With this definition, desirable properties hold. 
		In particular, 
		the original L$^\natural$-convex functions 
		coincide with L-convex functions on the product of directed paths, where $\ZZ$ is identified with an infinite directed path.
		}
\end{Rem}

\section*{Acknowledgments}
The author thanks Yuni Iwamasa for careful reading, Satoru Fujishige for remarks,  
and Kazuo Murota for numerous comments improving presentation.
The work was partially supported by JSPS KAKENHI Grant Numbers 25280004, 26330023, 26280004, 17K00029.
%

\end{document}